\documentclass{amsart}
\numberwithin{equation}{section}
\usepackage[utf8]{inputenc}
\usepackage[T1]{fontenc}
\usepackage{lmodern}
\usepackage[margin=1.33in]{geometry}
\usepackage{setspace}
\usepackage[usenames, dvipsnames]{xcolor}
\usepackage{graphicx}
\usepackage{mathtools}
\usepackage{amssymb}
\usepackage{amsthm}
\usepackage{xspace}
\usepackage{tikz-cd}
\usepackage{quiver}
\usepackage[backref=page, bookmarks=false]{hyperref}
\usepackage[capitalize, noabbrev]{cleveref}
\newcommand{\iso}{\cong}
\newcommand{\isoto}{\xrightarrow{\simeq}}

\newcommand{\Z}{\mathbb Z}

\newcommand{\R}{\mathbb R}

\newcommand{\T}{\mathbb T}
\newcommand{\F}{F}

\newcommand{\op}{\mathsf{op}}
\newcommand{\GL}{\mathrm{GL}}
\newcommand{\GLp}{\mathrm{GL}^+}

\newcommand{\PSL}{\mathrm{PSL}}
\newcommand{\Lie}{\mathrm{Lie}}

\newcommand{\Bdot}{B_\bullet}

\newcommand{\BdotG}{B_\bullet G}
\newcommand{\EdotG}{E_{\bullet}G}
\newcommand{\BdotH}{B_\bullet H}
\newcommand{\BnablaG}{B_\nabla G}
\newcommand{\EnablaG}{E_\nabla G}

\newcommand{\GFG}{\Gamma \backslash F / G}
\newcommand{\GFR}{\Gamma \backslash F / \R}
\newcommand{\GF}{\Gamma \backslash F}
\newcommand{\Diff}{\mathrm{Diff}^+(S^1)}

\newcommand{\Witt}{\mathfrak{w}}

\newcommand{\pr}{\mathrm{pr}}

\newcommand{\point}{\ast}

\DeclareMathOperator{\Hom}{Hom}

\newtheorem{lem}[equation]{Lemma}
\newtheorem{cor}[equation]{Corollary}
\newtheorem{prop}[equation]{Proposition}
\newtheorem{thm}[equation]{Theorem}
\newtheorem{theorem}[equation]{Theorem}
\newtheorem*{mainthm}{\cref{main_thm}}

\theoremstyle{definition}
\newtheorem{exm}[equation]{Example}
\newtheorem{defn}[equation]{Definition}

\theoremstyle{remark}
\newtheorem{rem}[equation]{Remark}

\crefname{thm}{Theorem}{Theorems}
\crefname{lem}{Lemma}{Lemmas}
\crefname{cor}{Corollary}{Corollaries}
\crefname{prop}{Proposition}{Propositions}
\crefname{ex}{Exercise}{Exercises}
\crefname{exm}{Example}{Examples}
\crefname{defn}{Definition}{Definitions}
\crefname{claim}{Claim}{Claims}
\crefname{rem}{Remark}{Remarks}
\crefname{fct}{Fact}{Facts}
\crefname{note}{Note}{Notes}

\newcommand{\term}{\emph} 

\DeclarePairedDelimiter\set{\{}{\}}
\DeclarePairedDelimiter\paren{(}{)}

\DeclarePairedDelimiter\ket{\lvert}{\rangle}

\makeatletter
	\let\oldparen\paren
	\def\paren{\@ifstar{\oldparen}{\oldparen*}}
\makeatother

\usepackage{microtype}
\usepackage{hypcap}

\hypersetup{
 colorlinks,
 linkcolor={red!50!black},
 citecolor={green!50!black},
 urlcolor={blue!80!black}
}

\newcommand{\cat}{\mathsf}
\newcommand{\Man}{\cat{Man}}
\newcommand{\Sh}{\cat{Sh}}
\newcommand{\Ch}{\cat{Ch}}
\newcommand{\sSet}{\cat{sSet}}

\newcommand{\Sym}{\mathrm{Sym}}
\newcommand{\fg}{\mathfrak g}
\newcommand{\fk}{\mathfrak k}
\newcommand{\gl}{\mathfrak{gl}}
\newcommand{\so}{\mathfrak{so}}
\newcommand{\SO}{\mathrm{SO}}
\newcommand{\tr}{\mathrm{tr}}

\renewcommand{\d}{\mathrm d}
\newcommand{\ud}{\,\d}
\renewcommand{\dh}{\d^h}
\newcommand{\dv}{\d^v}

\newcommand{\tGam}{\widetilde\Gamma}
\newcommand{\CExt}{\mathrm{CExt}}
\newcommand{\Map}{\mathrm{Map}}

\newcommand{\GpVir}{\widetilde\Gamma}

\newcommand{\bl}{\text{--}}
\newcommand{\fR}{\mathfrak{g}_\R}
\newcommand{\fu}{\mathfrak u}
\newcommand{\U}{\mathrm{U}}

\newcommand{\Inv}[1][*]{I^{#1}}
\newcommand{\ChW}{\mathit{CW}}

\newcommand{\Vir}{\mathit{\mathcal Vir}}

\title{Constructing the Virasoro groups using differential cohomology}
\author{Arun Debray}
\email{adebray@purdue.edu}

\author{Yu Leon Liu}
\email{yuleonliu@math.harvard.edu}

\author{Christoph Weis}
\email{weis@maths.ox.ac.uk}

\date{\today}

\begin{document}    

\begin{abstract}
The Virasoro groups are a family of central extensions of $\mathrm{Diff}^+(S^1)$, the group of
orientation-preserving diffeomorphisms of $S^1$, by the circle group $\mathbb T$. We give a novel, geometric construction of
these central extensions using ``off-diagonal'' differential lifts of the first Pontryagin class, thus affirmatively answering a question of Freed-Hopkins.
\end{abstract}
\maketitle

\tableofcontents

\setcounter{section}{-1}
\section{Introduction}
\label{sec:intro}

%
Two-dimensional conformal field theories (CFTs) are a Goldilocks zone in mathematical physics: high-dimensional
enough to admit many examples and a rich structure, but still mathematically tractable. This is especially true for their
groups of symmetries. In all other dimensions, the group of conformal symmetries is finite-dimensional, but in two
dimensions, there is an infinite-dimensional family of conformal symmetries: the group $\Gamma\coloneqq \Diff$ of orientation-preserving diffeomorphisms of a circle. $\Diff$ acts on
the Hilbert space $\mathcal H$ of states that a 2d CFT assigns to a circle, and many important representations of
$\Diff$ arise in this way. However, as is generally the case in quantum physics, these are merely
projective symmetries: if $\ket\psi\in\mathcal H$ and $\lambda\in\T$, where $\T$ denotes the unit complex numbers,
then the states $\ket\psi$ and $\lambda\ket\psi$ are physically equivalent. Therefore the group that actually acts,
which is called a \term{Virasoro group}, is a central extension of $\Diff$ by $\T$. There is an $\R$ worth of Virasoro extensions; which central extension we
obtain depends on the CFT we began with. 

The Virasoro group extensions $\tGam_\lambda$ of $\Diff$ are defined as follows: as spaces,
$\tGam_\lambda\cong\T\times\Gamma$. However, the multiplication is twisted: as a map $(\T\times\Gamma)\times
(\T\times\Gamma)\to \T\times\Gamma$, multiplication obeys the formula
\begin{equation}
	(z_1, \gamma_1), (z_2, \gamma_2) \longmapsto (z_1 \cdot z_2 \cdot B_\lambda(\gamma_1, \gamma_2),
	\gamma_1\circ\gamma_2),
\end{equation}
where $B\colon\Gamma\times\Gamma\to\T$ is the \term{Bott-Thurston cocycle} ~\cite{Bot77} 
\begin{equation}
	B_\lambda(\gamma_1, \gamma_2) \coloneqq \exp\left(-\frac{i\lambda}{48\pi}\int_{S^1} \log(\gamma_1'
	\circ\gamma_2) \ud(\log(\gamma_2'))\right). 
\end{equation}
See \S\ref{sec:central_extensions} for the details, and \cref{rem-on-witt} for the relationship to the Virasoro
algebras, which may be more familiar.

A similar story can happen for other groups. For example, there is a class of 2d CFTs called
\term{Wess-Zumino-Witten (WZW) models}, given by choosing a compact Lie group $G$ and an element $h\in H^4(BG;\Z)$,
which admit a projective symmetry for the (unbased) loop group $LG$ of $G$. The corresponding central extensions by
$\T$, called \term{Kac-Moody groups}, are common objects of study in representation theory.\footnote{The WZW models
still have a projective $\Diff$-action, and the formulas for the central extensions of $\Diff$ and of $LG$ that
appear for a particular choice of $h$ are related by the \term{Segal-Sugawara formula}; see~\cite{KZ84} for more
information.}

Brylinski-McLaughlin~\cite[\S 5]{BM94} give a geometric construction of the Kac-Moody central extensions of loop groups using differential
cohomology, and the goal of this paper is to do a similar construction to obtain the Virasoro central extensions of
$\Diff$. First we briefly sketch Brylinski-McLaughlin's construction. Let $\Z(n)$ denote the
$n$th \term{Deligne complex}~\cite[\S 2.2]{Del71}, a complex of sheaves on the site $\Man$ of smooth
manifolds given by
\begin{equation}
	\Z(n)\coloneqq \Big(\!
	\begin{tikzcd}
		0 & \Z & {\Omega^0} & {\Omega^1} & \dotsb & {\Omega^{n-1}} & 0
		\arrow[from=1-1, to=1-2]
		\arrow[from=1-2, to=1-3]
		\arrow[from=1-3, to=1-4]
		\arrow[from=1-4, to=1-5]
		\arrow[from=1-5, to=1-6]
		\arrow[from=1-6, to=1-7]
	\end{tikzcd}\!
\Big).
\end{equation}
For $M$ a smooth manifold, $H^n(M;\Z(n))$ is naturally isomorphic to the group $\check H^n(M)$ of degree-$n$
Cheeger-Simons differential characters.\footnote{Cheeger-Simons' indexing convention differs from ours, and both
conventions appear in the literature; we follow~\cite{ADH21}. In our convention, the characteristic class map is
degree-preserving, with signature $\check H^n(M)\to H^n(M;\Z)$.\label{off_by_one}} We are interested in
$H^k(M;\Z(n))$ when $k$ is not necessarily equal to $n$. For $H$ a Lie group, possibly infinite-dimensional, let
$\BdotH$ denote the classifying stack for principal $H$-bundles (see \cref{Bdot}). Then the group
of equivalence classes of central extensions of $H$ by $\T$ can be naturally and explicitly identified with
$H^3(\BdotH;\Z(1))$~\cite[Corollary 17.3.3]{ADH21} (here \cref{TExtZ1}). In the case at hand, $H = LG$ is the loop group of
$G$, where $G$ is a compact, connected finite-dimensional Lie group, and there is an equivalence of stacks $L\BdotG
\simeq \Bdot LG$. We will see in \cref{compact_off_diag} that for compact $G$, the truncation map $t\colon
\Z(n)\to\Z$ induces an isomorphism
\begin{equation}
	t\colon H^{2n}(\BdotG;\Z(n))\longrightarrow H^{2n}(\BdotG;\Z) = H^{2n}(BG;\Z),
\end{equation}
where $BG$ is the classifying space of $G$ in the usual sense. Thus the level $h\in H^4(BG;\Z)$ used to define the
WZW model refines to a class $\widetilde h\in H^4(\BdotG;\Z(2))$. Using the diagram
\begin{equation}
\label{eqn:ev_on_s1}
\begin{tikzcd}
S^1 \times L\BdotG \arrow[r, "q"] \arrow[d, "p"]
& \BdotG  \\
L\BdotG,
& 
\end{tikzcd}
\end{equation}
where $q$ is the evaluation map and $p$ is projection onto the second factor, we can pull back along $q$ and then push
forward along $p$. The latter operation is realized by integrating over the $S^1$ fiber. This defines a
\term{transgression map} for compact, connected $G$:
\begin{equation}
	\tau\colon H^{4}(\BdotG;\Z(2)) \longrightarrow H^4(S^1 \times L\BdotG; \Z(2))\longrightarrow H^3(\Bdot LG; \Z(1)).
\end{equation}
Brylinski-McLaughlin showed that $\tau(\widetilde h)$ recovers the Kac-Moody central extension of $LG$
at level $h$, which acts on the WZW theory associated to $G$ and $h$. 

Freed-Hopkins conjectured that a similar procedure could construct the Virasoro groups~\cite[Question
17.3.8]{ADH21}. Specifically, let $\GL_n^+(\R)$ be the group of invertible, orientation-preserving $n\times n$
matrices and begin with the first Pontryagin class $p_1\in H^4(B\GL_n^+(\R);\Z)$. As $\GL_n^+(\R)$ is not compact, lifting to
Deligne cohomology is not automatic, but by~\cite[\S 17.3]{ADH21} (here \cref{GLn_off_diag}) there is an affine
line of lifts $\hat{p}_1^\lambda \in H^4(\Bdot\GLp_n(\R);\Z(2))$, labeled by $\lambda \in \R$.
Furthermore, there is a distinguished lift $\hat{p}_1$ that satisfies the (differential) Whitney sum formula
(\cref{distinguishedlifts}).\footnote{The case $n=1$ is special: the 1st Pontryagin class of $\GLp_1(\R)$ is trivial and we have a one-dimensional vector space of differential lifts. See \cref{lineofliftsforGLp1}.}

Let $E\to \Bdot\Diff$ be the
universal circle bundle and $V\to E$ its vertical tangent bundle; $E$ is the stacky quotient $S^1/\Diff$.
Because the $\Diff$-action on $S^1$ is orientation-preserving, $V$ is oriented, so there is a classifying
map $q: E \rightarrow \Bdot \GL^+_1(\R)$, and we have the following diagram:
\begin{equation}
\begin{tikzcd}
E \arrow[r, "q"] \arrow[d, "p"]
& \Bdot \GL^+_1(\R) \\
\Bdot\Diff.
& 
\end{tikzcd}
\end{equation}
Once again, we can pull back and integrate over the fiber to give a map
\begin{equation}
	\int_{S^1}\circ\ q^*\colon H^4(\Bdot \GL^+_1(\R);\Z(2))\longrightarrow H^3(\Bdot\Diff;\Z(1)).
\end{equation}

So after choosing a lift $\hat{p}^\lambda_1$ of $p_1$, we obtain a central extension of $\Diff$ given by the class
$\int_{S^1} \hat{p}^\lambda_1(V)\in H^3(\Bdot\Diff;\Z(1))$ --- we suppress $q^\ast$ from notation. Freed-Hopkins' conjecture asserts that the family of
extensions of $\Diff$ given by all choices of $\hat{p}^\lambda_1$ is precisely the family of Virasoro extensions.
\begin{mainthm}
The transgression homomorphism
\begin{align}
    \begin{split}
        H^4(\Bdot\GL_1^+(\R);\Z(2)) &\to H^3(\Bdot\Diff;\Z(1)) \\
        \hat{p}^\lambda_1 &\mapsto \int_{S^1} \hat{p}^\lambda_1(V)
    \end{split}
\end{align}
maps the $\R$ worth of lifts of $p_1$ isomorphically to the $\R$ of Virasoro central extensions of $\Diff$. Furthermore, it takes the distinguished differential lift $\hat{p}_1$ to the Virasoro central extension $\tGam_{-12}$ with central charge $-12$.
\end{mainthm}

Our proof proceeds mostly at the cocycle level. First, though, we prove \cref{pullback_to_diff_forms}, that the
lifts of $p_1$ are in the image of a map $H^2(\Bdot\GL_1^+(\R);\Omega^1)\to H^4(\Bdot \GL_1^+(\R); \Z(2))$.
This allows us to compute the transgression map at the level of differential forms. To do so, we model the
universal circle bundle $E$ over $\Bdot\Diff$ as the stacky double quotient $\Diff\backslash F/\GL_1^+(\R)$,
where $F$ is the frame bundle of $S^1$. This double quotient has a bisimplicial presentation resolving both the
left $\Diff$- and right $\GL_1^+(\R)$-actions. It interpolates between the simplicial objects corresponding to the two actions. We chase the generator of the $\R$
worth of differential lifts across this double complex to obtain a form that is easier to integrate over the $S^1$
fibers. Then we integrate it and see that we obtain the Bott-Thurston cocycle.

\subsection*{Outline}
Our first few sections review information that we need to perform the computation. In
\S\ref{sec:central_extensions}, we introduce the Virasoro groups and the Bott-Thurston cocycles that define
them. We collect differential cohomology information in \S\ref{sec:diffcoh_background}, where we introduce the
Deligne complexes $\Z(n)$ and prove a few quick lemmas about them. In \S\ref{sec:off_diagonal}, we study lifts of
characteristic classes to $H^{2n}(\BdotG;\Z(n))$. The key fact in this section is \cref{botts_theorem}, a theorem
of Bott which allows one to compute $H^{2n}(\BdotG;\Z(n))$ in terms of the Chern-Weil homomorphism. We use this in
\cref{GLn_off_diag} to study the affine line of lifts of $p_1$ to $H^4(\Bdot\GL_n^+(\R);\Z(2))$, including the
distinguished lift $\hat p_1$ which satisfies the Whitney sum formula (\cref{distinguishedlifts}).

In
\S\ref{sec:cocycle} we compute explicit cocycles for the $\R$ worth of off-diagonal lifts of the first Pontryagin
class in $H^4(\Bdot \GL_1^+(\R); \Z(2))$. Importantly, \cref{pullback_to_diff_forms} allows us to work with
$\Omega^1$ instead of $\Z(2)$, simplifying the computation.
In \S \ref{sec:main_statement} we prove \cref{main_thm} by a computation through the bisimplicial object discussed above.


\subsection*{Acknowledgements}
Most of all we would like to thank Dan Freed and Mike Hopkins for sharing their question which inspired this paper,
and for many helpful mathematical discussions. This paper also benefited from conversations with Araminta Amabel, Sanath Devalapurkar, 
Peter Haine, André Henriques, Ralph Kauffman, Cameron Krulewski, Kiran Luecke, Natalia Pacheco-Tallaj, Wyatt Reeves, Charlie Reid, and Jim Stasheff; thank you to all.

\section{Central extensions of $\Diff$}
\label{sec:central_extensions}
    Let $\Gamma \coloneqq \Diff$, the group of orientation-preserving diffeomorphisms of the circle. This is a
\term{Fréchet Lie group}~\cite{Mil84}: it admits an atlas of charts valued in Fréchet spaces, and group
multiplication and inversion are Fréchet maps. The goal of this paper is to construct a particular family of
central extensions of $\Gamma$ called the Virasoro groups; in this section we discuss some basic information about
$\Gamma$ and its central extensions. See \cref{rem-on-witt} for the Lie algebra version of this story, which may be
more familiar.

We will also need to know about two subgroups of $\Gamma$. First, there is an inclusion
$\SO_2\subset\Gamma$ as rotations. We also have $i\colon \PSL_2(\R)\hookrightarrow\Gamma$ as the real
fractional linear transformations acting on $\mathbb{RP}^1 = S^1$.

For each $\lambda\in\R$, there is a Fréchet Lie group central extension of $\Gamma$ by the circle group $\T$ called
a \term{Virasoro group} and denoted $\tGam_\lambda$. As spaces, $\tGam_\lambda\cong\T\times\Gamma$.\footnote{In
fact, the inclusion $\SO_2\hookrightarrow\Gamma$ is a homotopy equivalence, so all principal $\T$-bundles
over $\Gamma$ are trivializable, including those coming from Fréchet Lie group central extensions.} However, the
multiplication is twisted: as a map $(\T\times\Gamma)\times (\T\times\Gamma)\to \T\times\Gamma$, it obeys the
formula
\begin{equation}
	(z_1, \gamma_1), (z_2, \gamma_2) \longmapsto (z_1 \cdot z_2 \cdot B_\lambda(\gamma_1, \gamma_2),
	\gamma_1\circ\gamma_2),
\end{equation}
where $B_\lambda\colon\Gamma\times\Gamma\to\T$ is the \term{Bott-Thurston cocycle} ~\cite{Bot77}:
\begin{equation}
\label{bcocyc}
	B_\lambda(\gamma_1, \gamma_2) \coloneqq \exp\left(-\frac{i\lambda}{48\pi}\int_{S^1} \log(\gamma_1'
	\circ\gamma_2) \ud(\log(\gamma_2'))\right). 
\end{equation}

The symbols in~\eqref{bcocyc} deserve further explanation. Given a diffeomorphism
$\gamma\colon S^1 \to S^1$ in $\Gamma$, 
we can lift it to a map $\tilde{\gamma}\colon \R \to \R$ along any covering map $\R \to S^1$.
The derivative $\tilde{\gamma}'\colon \R \to \R_+^\times \subset \Hom(\R,\R)$ 
descends to a function $\gamma'\colon S^1 \to \R_+^\times$ which is independent of the choices made.
(These maps land in $\R_+^\times$ because we are dealing with orientation-preserving diffeomorphisms.)
The function $\log\colon \R_+^\times \to \R$ used in~\eqref{bcocyc} is the natural
logarithm.

\begin{rem}
Bott's original formula for this cocycle was slightly different:
\begin{equation}
B_\lambda(\gamma_1, \gamma_2) \coloneqq \exp\left(-\frac{i\lambda}{48\pi} \int_{S^1} \log((\gamma_1\circ\gamma_2)')
\ud(\log(\gamma_2'))\right).
\end{equation}
This is equal to the cocycle in~\eqref{bcocyc} because $\log((\gamma_1\circ\gamma_2)') =
\log(\gamma_1'\circ\gamma_2) + \log(\gamma_2')$ and
\begin{equation}
\int_{S^1} \log(\gamma_2')\ud\log(\gamma_2') = 0  
\end{equation}
because $\log(\gamma_2')\ud(\log(\gamma_2')) = \frac{1}{2} \d(\log(\gamma_2'))^2$ is exact.
\end{rem}

\begin{rem}
\label{bosonic}
As we mentioned in the introduction, $\Gamma$ acts projectively on two-dimensional conformal field theories,
lifting to actual representations of the Virasoro groups. This allows us to choose a favorite
normalization: the constant $\tfrac{1}{48}$ is chosen so that $\tGam_1$ acts on the theory of bosonic periodic
scalars (strings). In physics terms, the normalization is set so the theory of bosonic scalars has central charge
$1$.
\end{rem}

\begin{rem} \label{vira-from-R-extension}
Let $H$ be a Fréchet Lie group. Then a central extension of $H$ by $\R$ gives rise to a central extension
$\widetilde
H$ of $H$ by $\T$ via pushout of groups:
\begin{equation}
\begin{tikzcd}
	{\R} & \widetilde H_\R\\
	\T & \widetilde H
	\arrow[from=1-1, to=1-2]
	\arrow[from=2-1, to=2-2]
	\arrow["\exp(2\pi i -)"', from=1-1, to=2-1]
	\arrow[from=1-2, to=2-2].
\end{tikzcd}
\end{equation}
For each $\lambda\in\R$, the Virasoro extension $\tGam_\lambda$ arises in this way from an extension of $\Gamma$ by
$\R$ defined using the $\R$-valued cocycle
\begin{equation}
\label{real-cocycle}
	B_{\lambda, \R}(\gamma_1, \gamma_2) \coloneqq -\frac{\lambda}{96\pi^2}\int_{S^1} \log(\gamma_1'
	\circ\gamma_2) \ud(\log(\gamma_2')). 
\end{equation}
\end{rem}

\begin{rem}
\label{rem-on-witt}
Fréchet Lie groups have a notion of Lie algebras, which are Fréchet spaces; the Lie algebra of $\Gamma$ is the
Fréchet space of smooth vector fields on $S^1$ with its usual bracket. This is the completion of a Lie
subalgebra $\Witt$ called the \term{Witt algebra}, which is the Lie algebra of polynomial vector fields on $S^1$.
The Witt algebra is generated by $L_n \coloneqq -i e^{i n\theta} \frac{\partial}{\partial \theta}$, $n\in\Z$, with
commutation relations
\begin{equation}
    [L_m, L_n] = (m-n)L_{m+n}.
\end{equation}

Differentiating a central extension of (Fréchet) Lie groups produces a central extension of (Fréchet) Lie algebras. 
Applied to the Virasoro extensions $\tGam_\lambda$, we obtain a family of central extensions
$\widetilde{\Witt}_\lambda$ of the Witt algebra $\Witt$ by $\R$, called \term{Virasoro algebras}.
The set of equivalence classes of central extensions of a Lie algebra $\fg$ by $\R$ is given by the Lie algebra
cohomology group $H^2(\fg; \R)$; in \cite{GF68}, Gel'fand-Fuks showed that $\CExt_\R(\Witt)\cong\R$, with
$\lambda\in\R$ corresponding to $\widetilde{\Witt}_\lambda$.

As vector spaces,
$\widetilde{\Witt}_\lambda\cong\R\times\Witt$, and we can find Lie algebra cocycles $b_\lambda\colon
\Witt\times\Witt\to\R$ such that the Lie bracket
$\widetilde{\Witt}_\lambda\times\widetilde{\Witt}_\lambda\to\widetilde{\Witt}_\lambda$ has the formula
\begin{equation}
	[(x_1, y_1), (x_2, y_2)] \coloneqq (b_\lambda(y_1, y_2), [y_1, y_2]).
\end{equation}
where $x_1,x_2\in\R$ and $y_1,y_2\in\Witt$. 
With our normalization, these cocycles are defined on generators by
\begin{equation}
\label{lie_alg_cocycle}
    b_\lambda(L_m, L_n) \coloneqq \frac{\lambda}{12} m^2(m-1)\delta_{m, -n}.
\end{equation}
Denote the central element generating the copy of $\R$ in $\widetilde{\Witt}_\lambda$ by $c_\lambda$. 
Its prefactor $\lambda$ in the commutator is called the \term{central charge}. The Virasoro algebras with $\lambda\ne 0$ are all
isomorphic to each other: explicitly, define $\widetilde{\Witt}_1\overset\simeq\to\widetilde{\Witt}_\lambda$ by
sending $L_m\mapsto L_m$ and $c_1\mapsto \lambda c_\lambda$; this is \emph{not} a map of central extensions, as a
map of central extensions must be the identity on the central elements. In addition, these
isomorphisms do not lift to isomorphisms of the corresponding Virasoro groups. Nonetheless, because of these
identifications, $\widetilde{\Witt}_1$ is sometimes referred to as \emph{the} Virasoro algebra.
\end{rem}
The set of equivalence classes of Fréchet Lie group central extensions of a Fréchet Lie group $G$ by $\T$ is an
abelian Lie group $\CExt_\T(G)$; in \cref{TExtZ1} we will explicitly identify it with a sheaf cohomology group. The
Virasoro central extensions define a subgroup $\Vir\subset\CExt_\T(G)$ isomorphic to $\R$.
\begin{thm}[{Segal~\cite[Corollary 7.5]{Seg81}}]
\label{segal_cext}
$\CExt_\T(\Gamma) \cong \CExt_\T(\PSL_2(\R))\times\Vir$.
\end{thm}
In particular, the Virasoro extensions are trivial when restricted to $\PSL_2(\R)$. The summand
$\CExt_\T(\PSL_2(\R))$ is isomorphic to $\T$~\cite[\S 7]{Seg81}, so
\begin{equation}
\label{segal_simple}
	\CExt_\T(\Gamma)\overset\cong \longrightarrow \T\times\R.
\end{equation}
\begin{rem}
If $\CExt_\R(\Witt)$ denotes the vector space of Lie algebra central extensions of $\Witt$ by $\R$, then
differentiation defines a map $d\colon \CExt_\T(\Gamma)\to\CExt_\R(\Witt)$. So a more refined version of
\cref{segal_cext} is that if $i\colon \PSL_2(\R)\hookrightarrow\Gamma$ is the inclusion map, then
\begin{equation}
\label{segal_better}
	(i^*, d)\colon\CExt_\T(\Gamma)\longrightarrow \CExt_\T(\PSL_2(\R)) \times \CExt_\R(\Witt)
\end{equation}
is an isomorphism. That is, a central extension of $\Gamma$ is uniquely characterized by its derivative and its
restriction to $\PSL_2(\R)$. The Virasoro central extensions are all trivializable when restricted to $\PSL_2(\R)$,
and (when $\lambda\ne 0$) are nontrivial on the level of Lie algebras. Thus the subgroup of Virasoro central
extensions is the subgroup $\set 0\times\R$ under the isomorphism~\eqref{segal_simple}.
\end{rem}

\begin{rem}
The central extensions defining the Virasoro algebras were independently discovered several times. First,
Block~\cite[\S 2]{Blo66} wrote down a version of~\eqref{lie_alg_cocycle} for a positive-characteristic analogue of
$\Witt$; then Gel'fand-Fuks~\cite{GF68} found cocycles for the Virasoro extensions in characteristic zero. The
Virasoro algebra extensions were then rediscovered in physics by Weis (see Brower-Thorn~\cite[\S 2]{BT71}).

Given a Lie algebra, it is natural to ask whether it can be exponentiated to a Lie group, and we are not sure who
was the first to ask this for the Virasoro algebras. The earliest reference we know of for a cocycle defining the
Virasoro group extension is Bott~\cite{Bot77}.
\end{rem}
\begin{rem}
Not everyone means the same thing by ``the Virasoro group(s).'' Some fix the normalization $\lambda = 1$. Others
define the Virasoro groups to be central extensions of $\Gamma$ by something different. For example, some authors
consider central extensions of $\Gamma$ by $\R$~\cite{Bot77, Lem95, Obl17}; others consider a simply connected
version, an extension of the universal cover of $\Gamma$ by $\R$~\cite{NS15}. Our interest in the Virasoro groups
is motivated by the projective $\Gamma$-symmetry in 2d conformal field theory, so we do not need to go any farther
than $\T$.
\end{rem}

\section{Background on differential cohomology}
\label{sec:diffcoh_background}
    The goal of this section is to set up our perspective on differential cohomology. For a more in-depth introduction
see~\cite{ADH21}.

Following Bunke-Nikolaus-Völkl~\cite{BNV16}, we think of differential cohomology in terms of sheaves of spectra on
the site $\Man$ of smooth manifolds. In this paper, we only need ordinary differential cohomology, which means we
just have to think about sheaves of chain complexes of abelian groups, or equivalently chain complexes of sheaves
of abelian groups.  Let $\Sh(\Man; \Ch)$ denote the category of chain complexes of sheaves of abelian groups. If
$M$ is a smooth manifold and $E\in\Sh(\Man; \Ch)$, then the $E$-cohomology of $M$, denoted $H^*(M; E)$, refers to
the hypercohomology of $M$ with coefficients in $E$.
\begin{defn}
Let $A$ be an abelian Lie group. We use $A$ to refer to the sheaf of abelian groups on $\Man$ whose value
on $M$ is the $A$-valued functions on $M$, where $A$ carries the discrete topology, and we use $\underline A$ to
denote the sheaf of $A$-valued functions where $A$ carries its usual topology.
\end{defn}
$\Omega^k$ denotes the sheaf of differential $k$-forms. Note $\underline\R\simeq\Omega^0$.
\begin{defn}
The \term{Deligne complex} $\Z(n)$~\cite[\S 2.2]{Del71} is defined as follows:
\begin{equation}
\label{deligne_defn}
	\Z(n)\coloneqq \Big(\!
	\begin{tikzcd}
		0 & \Z & {\Omega^0} & {\Omega^1} & \dotsb & {\Omega^{n-1}} & 0
		\arrow[from=1-1, to=1-2]
		\arrow[from=1-2, to=1-3]
		\arrow[from=1-3, to=1-4]
		\arrow[from=1-4, to=1-5]
		\arrow[from=1-5, to=1-6]
		\arrow[from=1-6, to=1-7]
	\end{tikzcd}\!
\Big)
\end{equation}
where the map $\Z\to\Omega^0$ realizes a $\Z$-valued function as an $\underline\R$-valued function, which is the
same thing as a $0$-form. The map $\Omega^k \to \Omega^{k+1}$ is given
by the exterior derivative.

We define $\R(n)$ analogously, with $\R$ (with the discrete topology) replacing $\Z$ in~\eqref{deligne_defn}. In particular, $\Z(0)$ and
$\R(0)$ are the sheaves $\Z$ and $\R$, respectively, whose cohomology is ordinary cohomology in $\Z$, resp.\ $\R$.
\end{defn}
\begin{exm}
We also use the complex
\begin{equation}
\label{T_deligne_defn}
	\T(n)\coloneqq \Big(\!
	\begin{tikzcd}
		0 & \underline\T & {\Omega^1} & \dotsb & {\Omega^n} & 0.
		\arrow["\varphi", from=1-2, to=1-3]
		\arrow[from=1-1, to=1-2]
		\arrow[from=1-3, to=1-4]
		\arrow[from=1-4, to=1-5]
		\arrow[from=1-5, to=1-6]
	\end{tikzcd}\!
\Big)
\end{equation}
where $\varphi$ is the map $(1/2\pi i)\ud\log$. 
Here, $\d\log:\T \to i\Omega^1$ maps a $\T$-valued function $f \in \T(M)$ to 
the differential form $\d \log (f) \coloneqq \tfrac{1}{f}\d f \in i\Omega^1(M)$, for all $M \in \Man$.

\end{exm}
\begin{lem}[{\cite[Remark 3.6]{BM94}}]
\label{lem:Tn_equiv_Zn}
There is an equivalence $\T(n)[-1]\simeq \Z(n+1)$.
\end{lem}
\begin{proof}
The map of complexes
\begin{equation}
\begin{tikzcd}
	0 & \Z & {\underline{\R}} & {\Omega^1} & \dotsb & {\Omega^n} & 0 \\
	{} & 0 & {\underline{\T}} & {\Omega^1} & \dotsb & {\Omega^n} & 0
	\arrow[from=1-1, to=1-2]
	\arrow[from=1-2, to=2-2]
	\arrow[from=2-2, to=2-3]
	\arrow[from=1-2, to=1-3]
	\arrow["\exp{2\pi i}"', from=1-3, to=2-3]
	\arrow["\d", from=1-3, to=1-4]
	\arrow["\frac{1}{2\pi i}\d\log", from=2-3, to=2-4]
	\arrow[Rightarrow, no head, from=1-4, to=2-4]
	\arrow[from=1-4, to=1-5]
	\arrow[from=2-4, to=2-5]
	\arrow[from=1-5, to=1-6]
	\arrow[from=2-5, to=2-6]
	\arrow[Rightarrow, no head, from=1-6, to=2-6]
	\arrow[from=1-6, to=1-7]
	\arrow[from=2-6, to=2-7]
	\arrow[Rightarrow, no head, from=1-5, to=2-5]
\end{tikzcd}
\end{equation}
provides an explicit equivalence. Here we have already used the fact that $\Omega^0\simeq\underline\R$ to identify
the top row with $\Z(n+1)$.
\end{proof}
\begin{rem}
Brylinski~\cite[Proposition 1.5.7]{Bry93} established a natural isomorphism between $H^n(M;\Z(n))$ and $\check H^n(M)$, the group of
Cheeger-Simons differential characters of $M$~\cite{CS85}.\footnote{As mentioned in Footnote~\ref{off_by_one}, we
use a different indexing convention than Cheeger-Simons.} So these ``diagonally-graded'' groups are the ordinary differential
cohomology groups of $M$. But through work of Beĭlinson~\cite{Bei84}, Brylinski~\cite{Bry99}, and others, it has become clear
that the ``off-diagonal'' groups $H^k(M;\Z(n))$, $k\ne n$, are also interesting. We will work primarily with
off-diagonal cohomology groups of the Deligne complexes.
\end{rem}
\begin{prop}[{Bunke-Nikolaus-Völkl~\cite[\S 4.1]{BNV16}, see also Hopkins-Singer~\cite[\S 3.2]{HS05}}]
Choose $0\le m<n$ and let $t\colon\Z(n)\to\Z(m)$ and $t\colon\R(n)\to\R(m)$ denote the truncation maps, which send
$\Omega^{i-1}$ in degree $i$ to $0$ for $i> m$ and do not change the terms in degrees $i\le m$. Then the commutative
square
\begin{equation}
\label{Znpullback}
\begin{tikzcd}
	{\Z(n)} & \Z(m) \\
	{\R(n)} & \R(m)
	\arrow["t", from=1-1, to=1-2]
	\arrow["t", from=2-1, to=2-2]
	\arrow[from=1-1, to=2-1]
	\arrow[from=1-2, to=2-2]
\end{tikzcd}
\end{equation}
is homotopy Cartesian.
\end{prop}
Letting $m = 0$, this gives us a square comparing differential cohomology to ordinary cohomology; this is the case
we use most often.

Differential cohomology comes with a fiber integration map: if $E\to B$ is a fiber bundle of manifolds with fiber
$S^1$ whose vertical tangent bundle is oriented, fiber integration is a map
\begin{equation}
	\int_{S^1}\colon H^k(E;\Z(n))\longrightarrow H^{k-1}(B;\Z(n-1)).
\end{equation}
Different constructions of this map have been given in~\cite{GT00, DL05, HS05, BKS10, Sch13, BB14, BNV16}. For this
paper, we only need to know one thing about this map: consider the diagram
\begin{equation}
\label{fib_int_comm}
\begin{tikzcd}
	{H^n(E; \Omega^k[-(k+1)])} & {H^n(E;\Z(k+1))} \\
	{H^{n-1}(B; \Omega^{k-1}[-k])} & {H^{n-1}(B;\Z(k))}
	\arrow["\varphi", from=1-1, to=1-2]
	\arrow["\varphi", from=2-1, to=2-2]
	\arrow["{\int_{S^1}}"', from=1-1, to=2-1]
	\arrow["{\int_{S^1}}"', from=1-2, to=2-2]
\end{tikzcd}
\end{equation}
given by comparing fiber integration with integration of differential forms. The horizontal maps come from a map
$\varphi\colon\Omega^k[-(k+1)]\to\Z(k+1)$, which is the fiber of the truncation map $\Z(k+1)\to\Z(k)$.
\begin{lem}
\label{fibcomm}
The diagram \eqref{fib_int_comm} commutes.
\end{lem}
This is because fiber integration on the part of the Deligne complex coming from differential forms is defined in
terms of integration of differential forms (see, e.g., \cite[\S 3.5]{HS05}).

We will also need to work with stacks on $\Man$. By a \term{stack} we mean a simplicial sheaf
$\Man^\op\to\sSet$. If $X_\bullet$ is a simplicial Fréchet manifold, it defines a stack $\mathbf X$ on $\Man$: the
value of this stack on a test manifold $M$ is the simplicial set $\Map(M, X_\bullet)$ whose $n$-simplices are the
set $\Map(M, X_n)$; see~\cite[Example 5.5]{FH13}. We say that $\mathbf X$ \term{is presented by} the simplicial
manifold $X_\bullet$. If $E\in\Sh(\Man; \Ch)$ and $\mathbf X$ is a stack presented by $X_\bullet$, then we define
$H^*(\mathbf X; E)$ to be the hypercohomology of the triple complex associated to $E^*(X_\bullet)$; this does not
depend on the choice of presentation of $\mathbf X$.

\begin{exm}
\label{X/G}
We are principally interested in group quotients. Let $G$ be a (Fr\'echet) Lie group and $X$ be a manifold with a smooth
right $G$-action. Then the quotient stack $X/G$ 
can be presented by the simplicial manifold
\begin{equation}
\label{XG_present}
\begin{tikzcd}[arrows={-Stealth}]
	X & 
	\ar[l,shift left=0.15em] \ar[l,shift right=0.15em]  X \times G&
	\ar[l,shift left=0.3em] \ar[l] \ar[l,shift right=0.3em] X\times G \times G &
	\ar[l,shift left=0.45em] \ar[l,shift left=0.15em]
	\ar[l,shift right=0.15em] \ar[l,shift right=0.45em] X\times G\times G \times G \dots
\end{tikzcd}
\end{equation}
The face maps are
\begin{equation}
	d_i(x, g_1, \dotsc, d_n) = \begin{cases}
		(x\cdot g_1, g_2, \dotsc, g_n) &i=0\\
		(x, g_1, \dotsc, g_ig_{i+1}, \dotsc, g_n), &0<i<n\\
		(x, g_1, \dotsc, g_{n-1}), &i = n.
	\end{cases}
\end{equation}
We will also need to take quotient stacks of left $G$-actions, producing a mirror-image diagram representing the
stacky quotient $G\backslash X$:
\begin{equation}
\begin{tikzcd}[arrows={-Stealth}]
	X & 
	\ar[l,shift left=0.15em] \ar[l,shift right=0.15em]  G \times X&
	\ar[l,shift left=0.3em] \ar[l] \ar[l,shift right=0.3em] G\times G \times X &
	\ar[l,shift left=0.45em] \ar[l,shift left=0.15em]
	\ar[l,shift right=0.15em] \ar[l,shift right=0.45em] G\times G\times G \times X \dots,
\end{tikzcd}
\end{equation}
whose face maps are
\begin{equation}
	d_i(g_1, \dotsc, d_n, x) = \begin{cases}
		(g_2, \dotsc, g_n, x) &i=0\\
		(g_1, \dotsc, g_ig_{i+1}, \dotsc, g_n, x), &0<i<n\\
		(g_1, \dotsc, g_{n-1}, g_n\cdot x), &i = n.
	\end{cases}
\end{equation}
\end{exm}
%
\begin{exm}
\label{Bdot}
When $X = *$, the quotient stack $\BdotG \coloneqq */G$ is the classifying stack for principal
$G$-bundles.\footnote{If $BG$ is a model for the classifying space of $G$, then there is another stack on $\Man$
given by the sheaf $\mathrm{Map}(\bl, BG)$. This is not the same as $\BdotG$.} That is, its value on a test
manifold $U$ is the nerve of the groupoid of principal $G$-bundles on $U$.  Using~\eqref{XG_present}, $\BdotG$ has
the simplicial presentation
\begin{equation}
\begin{tikzcd}[arrows={-Stealth}]
	\point & 
	\ar[l,shift left=0.15em] \ar[l,shift right=0.15em]  G &
	\ar[l,shift left=0.3em] \ar[l] \ar[l,shift right=0.3em] G\times G &
	\ar[l,shift left=0.45em] \ar[l,shift left=0.15em]
	\ar[l,shift right=0.15em] \ar[l,shift right=0.45em] G\times G\times G \dots
\end{tikzcd}
\end{equation}
We could have just as well made $G$ act on $*$ from the left; we would obtain the same simplicial manifold: the
identity map on $n$-simplices defines an equivalence $G\backslash *\to */G$.
\end{exm}
\begin{exm}
\label{Edot}
Likewise, let $X = G$ with $G$-action given by right multiplication. 
The corresponding quotient stack is known as $\EdotG$; it is equivalent to $G/G \simeq \ast$.
$\EdotG$ is also the classifying stack of trivial
principal $G$-bundles; that is,
its value on a test manifold $U$ is the nerve of the groupoid of trivialized
principal $G$-bundles on $U$. Forgetting the trivialization defines a map $\EdotG\to\BdotG$, and this map is the
universal principal $G$-bundle in the setting of stacks on $\Man$. Using~\eqref{XG_present}, $\EdotG$ has the
simplicial presentation
\begin{equation}
\begin{tikzcd}[arrows={-Stealth}]
	G & 
	\ar[l,shift left=0.15em] \ar[l,shift right=0.15em]  G \times G&
	\ar[l,shift left=0.3em] \ar[l] \ar[l,shift right=0.3em] G\times G \times G &
	\ar[l,shift left=0.45em] \ar[l,shift left=0.15em]
	\ar[l,shift right=0.15em] \ar[l,shift right=0.45em] G\times G\times G \times G \dots.
\end{tikzcd}
\end{equation}
\end{exm}
If $A$ is an abelian group considered with the discrete topology, then $H^*(\BdotG; A)\cong H^*(BG; A)$ more or
less by definition, and likewise $H^*(\EdotG; A)\cong H^*(EG; A) = A$ concentrated in degree $0$.
\begin{exm}
\label{Bnabla}
If $G$ is a finite-dimensional Lie group, there are analogues of $\BdotG$ and $\EdotG$ for principal $G$-bundles
with connection. $\BnablaG$ is the stack whose value on a test manifold $U$ is the nerve of the
groupoid of principal $G$-bundles with connection on $U$, and $\EnablaG$ is the stack whose value on $U$ is the
nerve of the groupoid of trivialized principal $G$-bundles with connection on $U$. Forgetting the trivialization
defines a map $\EnablaG\to\BnablaG$, which is the universal principal $G$-bundle with connection; in particular,
$\BnablaG$ is the quotient of $\EnablaG$ by a right $G$-action. There is an equivalence
$\EnablaG\simeq\Omega^1\otimes\fg$~\cite[(5.15)]{FH13}; under this equivalence, the $G$-action on $\EnablaG$ is by
gauge transformations:
\begin{equation}
    A \cdot g = g^{-1}Ag + g^{-1}\ud g.
\end{equation}
Freed-Hopkins~\cite[Proposition 5.24]{FH13} show that since $\BnablaG = \EnablaG/G \simeq (\Omega^1\otimes\fg)/G$,
$\BnablaG$ is equivalent to the simplicial object
\begin{equation}
\begin{tikzcd}[arrows={-Stealth}]
	\Omega^1\otimes\fg & 
	\ar[l,shift left=0.15em] \ar[l,shift right=0.14em]  (\Omega^1\otimes\fg)\times G &
	\ar[l,shift left=0.3em] \ar[l] \ar[l,shift right=0.3em] (\Omega^1\otimes\fg)\times G\times G &
	\ar[l,shift left=0.45em] \ar[l,shift left=0.15em]
	\ar[l,shift right=0.15em] \ar[l,shift right=0.45em] (\Omega^1\otimes\fg)\times G\times G\times G \dots
\end{tikzcd}
\end{equation}
\end{exm}
Because we are interested in extensions of $\Diff$, we allow $G$ to be an infinite-dimensional Fréchet Lie
group in \cref{Bdot}.
Brylinski~\cite[Proposition 1.6]{Bry00} showed that when $G$ is a Fréchet Lie group and $A$ is an abelian Lie
group, $H^2(\BdotG; \underline A)$ is naturally isomorphic to the group of equivalence classes of central
extensions of Fréchet Lie groups\footnote{There are several other notions of continuous or
smooth cohomology such that $H^2$ correctly classifies central extensions of topological or Lie groups, including
theories due to
Segal-Mitchison~\cite{Seg70, Seg75}, Wigner~\cite{Wig73}, Moore~\cite{Moo76}, Flach~\cite{Fla08},
Fuchssteiner-Wockel~\cite{FW12},
Khedekar-Rajan~\cite{KR12}, and Wagemann-Wockel~\cite{WW15}. Wagemann-Wockel (\textit{ibid.}, Theorem IV.5) provide
a general isomorphism theorem identifying most of these cohomology theories. Contrast this with the notion of globally continuous or smooth cohomology discussed for example in Stasheff~\cite{Sta78}.}
\begin{equation}
    \begin{tikzcd}
	0 & A & {\widetilde G} & G & 0.
	\arrow[from=1-1, to=1-2]
	\arrow[from=1-2, to=1-3]
	\arrow[from=1-3, to=1-4]
	\arrow[from=1-4, to=1-5]
\end{tikzcd}
\end{equation}
The Virasoro group is a central extension of $\Diff$ by $\T$, so we will be interested in $\underline{\T}$-cohomology. 
\begin{lem}
\label{TExtZ1}
For $G$ a Fréchet Lie group, equivalence classes of Fréchet Lie group central extensions by $\T$ are naturally
identified with $H^3(\BdotG; \Z(1))$.
\end{lem}
\begin{proof}
Apply \cref{lem:Tn_equiv_Zn} for $n = 0$, producing an equivalence $\underline{\T}[-1]\simeq\Z(1)$, hence an
equivalence $H^2(\BdotG;\underline{\T})\cong H^3(\BdotG;\Z(1))$.
\end{proof}
\begin{rem}
As spaces, Fréchet Lie group central extensions by $\T$ must be principal $\T$-bundles. Forgetting the central
extension and just remembering the principal $\T$-bundle over $G$ defines a homomorphism $H^3(\BdotG;\Z(1))\to
H^1(G;\underline{\T}) \overset\cong\to H^2(G;\Z)$. This map can also be described as follows: first apply the
truncation $\Z(1)\to\Z$ to land in $H^3(BG;\Z) \cong [BG, K(\Z, 3)]$; then take the loop space functor to land in
$[\Omega BG, \Omega K(\Z, 3)] \cong [G, K(\Z, 2)] \cong H^2(G;\Z)$.
\end{rem}
\begin{rem}
There is a pleasing interpretation of $H^4(\text{--};\Z(n))$ in terms of bundle gerbes with
various notions of connection due to Brylinski-McLaughlin~\cite{BM94}, Brylinski~\cite{Bry94, Bry99a},
Murray-Stevenson~\cite{MS00}, and Waldorf~\cite{Wal10}.
\end{rem}
%
%
%
%
%

\section{Off-diagonal differential characteristic classes}
\label{sec:off_diagonal}
    The goal of this section is to study lifts of characteristic classes to the off-diagonal differential cohomology
groups $H^{2n}(\BdotG; \Z(n))$. We will show that when $G$ is compact, the map $t\colon H^{2n}(\BdotG;\Z(n))\to
H^{2n}(BG;\Z)$ is an isomorphism (\cref{compact_off_diag}), and that in general (\cref{CW_pull_cor}) there is a
pullback square
\begin{equation}
   \begin{tikzcd}
	{H^{2n}(\BdotG; \Z(n))} & {H^{2n}(BG;\Z)} \\
	{\Sym^*(\fg^\vee)^G} & {\Sym^n(\fk^\vee)^K,}
	\arrow[from=1-2, to=2-2]
	\arrow["t", from=1-1, to=1-2]
	\arrow["j^*", from=2-1, to=2-2]
	\arrow[from=1-1, to=2-1]
\end{tikzcd}
\end{equation}
where $K$ is a maximal compact in $G$, $\fk$ is the Lie algebra of $K$, and $j^*$ is the pullback on functions
induced by the inclusion $j: \fk \hookrightarrow \fg$. We then use this to characterize differential lifts of
$p_1\in H^4(B\GL_n^+(\R);\Z)$.

First, though, we review Chern-Weil theory from a perspective that will be convenient when we study off-diagonal
characteristic classes. Let $G$ be a Lie group with $\pi_0(G)$ finite, and let $\fg$ be the Lie algebra of $G$.
\begin{defn}
We let $\Inv(G)$ denote the algebra $\Sym^*(\fg^\vee)^G$ of $G$-invariant polynomials on $\fg$, where $G$ acts by
the adjoint action.
\end{defn}
\begin{defn}[\cite{Car49, Che52}]
Let $M$ be a manifold and $P\to M$ be a principal $G$-bundle with connection $\Theta\in\Omega_P^1(\fg)$; let
$\Omega\in\Omega_P^2(\fg)$ be the curvature of $\Theta$. The \term{Chern-Weil map}
\begin{equation}
	\ChW\colon \Inv(G)\longrightarrow \Omega^*(M)
\end{equation}
is defined as follows: given $f\in\Inv[k](G)$, we can evaluate $f$ on $\Omega^{\wedge
k}\in\Omega_P^{2k}(\fg^{\otimes k})$, giving an element $f(\Omega^{\wedge k})\in\Omega_P^{2k}(\R)$; because $f$ is
$\mathrm{Ad}$-invariant, $f(\Omega^{\wedge k})$ descends to a form on $M$, and we define $\ChW(f)\in\Omega^{2k}(M)$
to be that form.
\end{defn}
This map is natural in $G$ and in $(P, \Theta)$. Since $\Omega$ is a $2$-form, this map doubles the grading.
\begin{thm}[Chern~\cite{Che52}, Weil~\cite{Wei49}]
$\ChW(f)$ is always a closed form, and its de Rham class does not depend on $\Theta$. $\ChW$ passes to an algebra homomorphism $\ChW\colon \Inv(G)\to H^*(BG;\R)$.
When $G$ is compact, this map is an isomorphism.
\end{thm}
The Chern-Weil map refines to on-diagonal differential cohomology:
\begin{thm}[{Cheeger-Simons~\cite[Theorem 2.2]{CS85}}]
\label{on_diagonal}
For $G$ a compact Lie group and $c\in H^{2n}(BG;\Z)$, there is a unique lift of $c$ to a class $\check c\in
H^{2n}(\BnablaG;\Z(2n))$ whose curvature form is the Chern-Weil form associated to the image of $c$ in $\R$-valued
cohomology.
\end{thm}
\begin{rem}
Cheeger-Simons did not work with the universal object $\BnablaG$, but instead with approximations of it.
Bunke-Nikolaus-Völkl~\cite[\S 5.2]{BNV16} showed Cheeger-Simons' construction lifts to $\BnablaG$.
Freed-Hopkins~\cite{FH13} interpret the Chern-Weil map as providing an isomorphism
$\Inv(G)\overset\simeq\to\Omega^*(\BnablaG)$.
\end{rem}
Beginning with work of Beĭlinson~\cite[\S 1.7]{Bei84}, there is a parallel and different story lifting $c$ to
$\tilde c\in H^{2n}(\BdotG;\Z(n))$, not requiring connections, and landing in off-diagonal differential cohomology
groups.

In this section we will study the truncation map $t\colon H^{2n}(\BdotG;\Z(n))\to H^{2n}(\BdotG;\Z) =
H^{2n}(BG;\Z)$. When $G$ is compact, this map is an isomorphism (\cref{compact_off_diag}). We are most interested
in $G = \GL_n^+(\R)$, which is not compact, but we will be able to characterize the preimages of the first Pontryagin
class $p_1\in H^4(B\GL_n^+(\R);\Z)$ in \cref{GLn_off_diag}.
\begin{rem}
Work on characteristic classes in off-diagonal Deligne cohomology began in the setting of algebraic or complex
geometry; this includes Beĭlinson's original work~\cite[\S 1.7]{Bei84} as well as work of Bloch~\cite{Blo78},
Soulé~\cite{Sou89}, Brylinski~\cite{Bry99, Bry99a}, and Dupont-Hain-Zucker~\cite{DHZ00} studying relationships
between off-diagonal and on-diagonal differential lifts of Chern classes. In the smooth setting, work began with
Bott's calculation of $H^*(\BdotG;\Omega^q)$~\cite{Bot73}, later reinterpreted in off-diagonal differential
cohomology by Waldorf~\cite{Wal10} and in~\cite[Chapters 15--17]{ADH21}. See also work of Shulman~\cite{Shu72} and
Bott-Shulman-Stasheff~\cite{BSS76} studying $H^*(\BdotG;\Omega^{\ge q})$.
\end{rem}
\begin{lem}[{\cite[\S 16.1]{ADH21}}] \label{lem:idk}
\label{2n_n_pullback}
Let $G$ be a finite-dimensional Lie group with $\pi_0(G)$ finite. Then there is a pullback square
\begin{equation}
\label{2n_n_eqn}
   \begin{tikzcd}
	{H^{2n}(\BdotG; \Z(n))} & {H^{2n}(BG;\Z)} \\
	{H^{2n}(\BdotG; \R(n))} & {H^{2n}(BG;\R),}
	\arrow[from=1-2, to=2-2]
	\arrow["t", from=1-1, to=1-2]
	\arrow["t", from=2-1, to=2-2]
	\arrow[from=1-1, to=2-1]
\end{tikzcd}
\end{equation}
where the horizontal maps are induced from the maps in~\eqref{Znpullback}, where $m = 0$, and we implicitly use the
identification $H^*(\BdotG; A) \simeq H^*(BG; A)$.
\end{lem}
\begin{proof}
The pullback square~\eqref{Znpullback} with $m = 0$ induces a Mayer-Vietoris sequence for the cohomology of $\BdotG$
with coefficients in $\Z$, $\R$, $\Z(n)$, and $\R(n)$. For any finite-dimensional Lie group $G$ with $\pi_0(G)$
finite, $H^{2n-1}(BG;\R) = 0$: retract $G$ onto its maximal compact; then, for compact Lie groups, the Chern-Weil
map $\Inv(G)\to H^*(BG;\R)$ is an isomorphism, and its image vanishes in odd degrees. Since $H^{2n-1}(BG;\R)
= 0$, the Mayer-Vietoris sequence simplifies into pullback squares of the form~\eqref{2n_n_eqn}.
\end{proof}

\Cref{lem:idk} tells us that if we want to lift characteristic classes from $H^{2n}(BG;\Z)$ to
$H^{2n}(\BdotG;\Z(n))$, it suffices to understand the map  $H^{2n}(\BdotG; \R(n)) \rightarrow H^{2n}(\BdotG; \R)$.
Note that we have a fiber sequence of sheaves
\begin{equation}
    \Omega^{k}[-(k+1)] \longrightarrow \R(k+1) \overset{t}{\longrightarrow} \R(k);
\end{equation}
it will suffice to understand $H^{i}(\BdotG; \Omega^j)$. Bott computes this in~\cite{Bot73}:
\begin{thm}[Bott~\cite{Bot73}]
\label{botts_theorem}
There is a natural ring isomorphism
\begin{equation}
    H^p(\BdotG; \Omega^q) \cong H^{p-q} _{\mathit{sm}} (G; \Sym^q(\fg^\vee)),
\end{equation}
where $H^*_{\mathit{sm}}(G;\Sym^q(\fg^\vee))$ is the \term{smooth cohomology} of $G$~\cite{HM62}.
\end{thm}

\begin{rem}
We care about the case where $p = q = n$, where \cref{botts_theorem} states that 
\begin{equation}
\label{2q_bott}
H^{n}(\BdotG; \Omega^n) \cong H_{\mathit{sm}}^0(G; \Sym^n(\fg^\vee)) = \Sym^n(\fg^\vee)^G = \Inv[n](G).
\end{equation}
\end{rem}
\begin{thm}[Bott~\cite{Bot73}]
\label{botts_theorem2}
There is a natural isomorphism $\phi\colon \Inv[n](G)\overset\cong\to H^{2n}(\BdotG;\R(n))$ such that the
composition
\begin{equation}
	\Inv[n](G)\overset\phi\longrightarrow H^{2n}(\BdotG;\R(n))\overset t\longrightarrow H^{2n}(\BdotG;\R)
	= H^{2n}(BG;\R)
\end{equation}
is the Chern-Weil homomorphism.
\end{thm}
This is not exactly the same as the theorem Bott gave in~\cite{Bot73}; see~\cite[Corollaries 16.2.4 and
16.2.5]{ADH21} for a proof of this version. 
By Chern-Weil theory, we know how to compute $H^{2n}(BG; \R)$: let $j\colon K\hookrightarrow G$ be a maximal compact subgroup and $\fk$ be the
Lie algebra of $K$. Then the composition
\begin{equation}
	\Inv(G)\overset{\ChW}{\longrightarrow} H^{2n}(BG;\R)\overset{j^*}{\longrightarrow}
	H^{2n}(BK;\R) \overset{\ChW^{-1}}{\longrightarrow} \Inv(K)
\end{equation}
can be identified with the map 
$\Inv(G)\to\Inv(K)$ induced from the inclusion of Lie algebras $\fk\subset\fg$.
Combine this with \cref{2n_n_pullback} to conclude:
\begin{cor}
\label{CW_pull_cor}
Let $G$ be a Lie group with $\pi_0(G)$ finite. let $j\colon K\hookrightarrow G$ be a maximal compact subgroup and $\fk$ be the
Lie algebra of $K$. Then there is a pullback square
\begin{equation}
\label{CW_pullback}
   \begin{tikzcd}
	{H^{2n}(\BdotG; \Z(n))} & {H^{2n}(BG;\Z)} \\
	{\Inv(G)} & {\Inv(K),}
	\arrow[from=1-2, to=2-2]
	\arrow["t", from=1-1, to=1-2]
	\arrow["j^*", from=2-1, to=2-2]
	\arrow[from=1-1, to=2-1]
\end{tikzcd}
\end{equation}
where $t$ is truncation and $j^*$ is the pullback of functions induced by the inclusion of Lie algebras $j:
\fk\subset\fg$.
\end{cor}
\begin{cor}
\label{compact_off_diag}
If $G$ is a compact Lie group, the map $t\colon H^{2n}(\BdotG;\Z(n))\to H^{2n}(BG;\Z)$ is an isomorphism.
\end{cor}
That is, off-diagonal differential lifts of characteristic classes exist and are unique for compact Lie groups. For
general noncompact $G$, off-diagonal differential lifts are not unique. We are interested in the case of $G =
\GLp_n(\R)$, the group of orientation-preserving invertible $n\times n$ matrices. Characteristic classes for
$\GL_n^+(\R)$ correspond to characteristic classes of rank-$n$ oriented real vector bundles.
\begin{cor}[{\cite[\S 17.3]{ADH21}}]
\label{GLn_off_diag}
Let $n\ge 2$. The space of lifts of the first Pontryagin
class $p_1\in H^4(B\GL_n^+(\R);\Z)$ across the truncation map
\begin{equation}
    t\colon H^4(B_\bullet\GL_n^+(\R); \Z(2)) \longrightarrow H^4(B_\bullet\GL_n^+(\R); \Z) = H^4(B\GL_n^+(\R); \Z)
\end{equation}
is a one-dimensional affine space. The image of these differential lifts under the map
\begin{equation}
	H^4(B_\bullet\GL_n^+(\R); \Z(2))\to H^4(B_\bullet\GL_n^+(\R); \R(2))\overset\cong\to
	\Inv[2](\GL_n^+(\R))
\end{equation}
is the affine space $\set{\lambda \tr(A)^2 - \frac{1}{8 \pi^2}\tr(A^2)\mid \lambda\in\R}$. Here,
$\tr: \gl_n^+(\R) \to \R$ is the usual trace map.
\end{cor}

\begin{proof}
Take the maximal compact $\SO_n \subset \GL_n^+(\R)$. The vector space $\Inv[2](\SO_n)$ is one-dimensional,
generated by $\tr(A^2)$. On the other hand, $\Inv[2](\GL_n^+(\R))$ is two-dimensional, with basis $\tr(A^2)$ and
$\tr(A)^2$. The natural map 
\begin{equation}
    j^* \colon \Inv[2](\GL_n^+(\R)) \longrightarrow \Inv[2](\SO_n)
\end{equation}
is surjective with kernel $\tr(A)^2$, since the elements of $\so_n$ are traceless.
By \cref{CW_pull_cor}, we have a pullback diagram
\begin{equation}
   \begin{tikzcd}
	{H^{4}(\Bdot \GL_n^+(\R); \Z(2))} & {H^{4}(B\GL_n^+(\R);\Z)} \\
	{\Inv[2](\GL_n^+(\R))} & {\Inv[2](\SO_n).}
	\arrow[from=1-2, to=2-2]
	\arrow["t", from=1-1, to=1-2]
	\arrow["j^*", from=2-1, to=2-2]
	\arrow[from=1-1, to=2-1]
\end{tikzcd}
\end{equation}
Recall that $p_1 \in H^{4}(B\GL_n^+(\R);\Z) \cong \Z$ is a generator, and its image under the Chern-Weil
homomorphism is $-\frac{1}{8\pi^2} \tr(A^2)$. Since $j^*$ is surjective with a one-dimensional kernel, the set
$(j^*)^{-1}(-\frac{1}{8\pi^2}\tr(A^2))$ is an affine line. Therefore $t^{-1}(p_1)$ is also an affine line, and its
image in $\Inv[2](\GL_n^+(\R))$ is precisely $(j^*)^{-1}(-\frac{1}{8\pi^2}(\tr(A^2)))$, which is the line of
polynomials of the form $-\frac{1}{8\pi^2}\tr(A^2) + q$ for $q\in\ker(j^*)$.
This is the affine line $\set{\lambda \tr(A)^2 - \frac{1}{8 \pi^2}\tr(A^2)\mid \lambda\in\R}$.
\end{proof}

When $n = 1$, $\GLp_1(\R)$ is homotopy equivalent to a point and $B\GLp_1(\R) \simeq *$. Therefore
$H^4(B\GLp_1(\R); \Z) \simeq H^4(*; \Z) \simeq 0$. The first Pontryagin class is $0 \in H^4(\Bdot\GLp_1(\R); \Z)$.
On the other hand, the stack $\Bdot\GLp_1(\R)$ is nontrivial. We declare that the differential lifts of the first
Pontryagin class are the preimages of $0$ under the map 
\begin{equation}
H^4(\Bdot\GL_1^+(\R); \Z(2)) \to  H^4(B\GLp_1(\R); \Z) = 0,   
\end{equation}
that is, all of $H^4(\Bdot\GL_1^+(\R); \Z(2))$. 

\begin{cor}
\label{lineofliftsforGLp1}
For $\GLp_1(\R)$, $H^4(\Bdot\GL_1^+(\R); \Z(2))$ is a one-dimensional $\R$ vector space, so there is a
one-dimensional vector space of lifts of the first Pontryagin class. The image of this vector space under the map
\begin{equation}
	H^4(B_\bullet\GL_1^+(\R); \Z(2))\to H^4(B_\bullet\GL_1^+(\R); \R(2))\overset\cong\to
	\Inv[2](\GL_1^+(\R))
\end{equation}
is the one-dimensional vector space $\Inv[2](\GL_1^+(\R)) = \set{\lambda A^2 \mid \lambda \in \R}$.
\end{cor}
\begin{proof}
Applying \cref{CW_pull_cor} to $\GLp_1(\R)$, we see that
\begin{equation}
H^4(\Bdot\GL_1^+(\R); \Z(2)) \longrightarrow \Inv[2](\GL_1^+(\R))
\end{equation} 
is an isomorphism. As $\Inv[2](\GL_1^+(\R))$ is a one-dimensional vector space generated by the polynomial $A^2$,
we have a whole line of differential lifts of the trivial first Pontryagin class $p_1 = 0 \in H^4(B\GLp_1(\R);
\Z)$.
\end{proof}
\begin{rem}
\label{compatible_family}
The first Pontryagin classes are defined for all $\GLp_n(\R)$ for $n \geq 1$. By abuse of notation we denote all of
them by $p_1$. They form a compatible family of characteristic classes in the following sense: for $n \leq m$,
consider the inclusion $\GLp_n(\R) \subset \GLp_m(\R)$ coming from the inclusion $\R^n\subset\R^m$ as the first $n$
coordinates; then the pullback of $p_1 \in H^4(B\GLp_m(\R); \Z)$ along the map
\begin{equation}
    H^4(B\GLp_m(\R); \Z) \longrightarrow H^4(B\GLp_n(\R); \Z)
\end{equation}
is the first Pontryagin class $p_1 \in H^4(B\GLp_n(\R); \Z)$. Alternatively, we can think about $p_1$ as a class in
$H^4(B\GL_\infty^+(\R); \Z)$. 

This family of classes satisfies the \emph{Whitney sum formula}: for all oriented real vector bundles 
$E,F \to X$,
\begin{equation}
\label{WSF}
    p_1(E \oplus F) = p_1(E) + p_1(F) 
\end{equation}
in $H^4(X;\Z)$.
Here there is an important subtlety: in general, the Whitney sum formula for Pontryagin classes only holds modulo
$2$-torsion. But for orientable vector bundles,~\eqref{WSF} does hold integrally: Brown~\cite[Theorem 1.6]{Bro82}
showed that for arbitrary vector bundles, the difference $p_1(E\oplus F) - p_1(E) - p_1(F)$ can be expressed in
terms of $w_1(E)$ and $w_1(F)$ (see also~\cite{Tho62}), so the difference vanishes when $E$ and $F$ are orientable.
So it is crucial that we are working with $\GLp_n(\R)$, and with $p_1$ --- even for oriented vector bundles, the
Whitney sum formula does not hold in general for higher-degree Pontryagin classes.
\end{rem}
The differential lifts of the first Pontryagin class also form a one-dimensional affine family of compatible
classes labeled by $\lambda \in \R$: given $\lambda$, the compatible classes $\hat{p}_1^\lambda \in H^4(\Bdot
\GLp_n(\R); \Z(2))$ are defined as follows:
\begin{enumerate}
    \item For $n \geq 2$, the image of $\hat{p}_1^\lambda$ under the map 
    \begin{equation}
    H^4(B_\bullet\GL_n^+(\R); \Z(2)) \longrightarrow \Inv[2](\GL_n^+(\R))
	\end{equation}
	is $\lambda \tr(A)^2 - \frac{1}{8 \pi^2}\tr(A^2)$. 
	\item For $n = 1$, the image of $\hat{p}_1^\lambda$ under the map 
	\begin{equation}
	    H^4(\Bdot\GL_1^+(\R); \Z(2)) \longrightarrow \Inv[2](\GL_1^+(\R))
	\end{equation}
	is $(\lambda - \frac{1}{8\pi^2})A^2$.
\end{enumerate}
This is a compatible family in the sense of \cref{compatible_family}: these classes pull back to one another under
the maps $\Bdot\GL_m^+(\R)\to\Bdot\GL_n^+(\R)$. Both $\tr(A^2)$ and $\tr(A)^2$ in $H^4(\Bdot\GLp_n(\R); \Z(2))$
pull back to $A^2$ in $H^4(\Bdot\GLp_1(\R); \Z(2))$.

We can ask for a family of lifts $\hat{p}_1^\lambda$ to satisfy the \emph{differential Whitney sum formula}: that
for all oriented real vector bundles $E,F \to X$, the equation 
\begin{equation}
\label{diffWhitney}
    \hat{p}_1^\lambda(E \oplus F) = \hat{p}_1^\lambda(E) + \hat{p}_1^\lambda(F) 
\end{equation}
holds in $H^4(X;\Z(2))$.

\begin{lem}
\label{distinguishedlifts}
    There is a unique family of lifts of $p_1$ that satisfy~\eqref{diffWhitney}. This family is the $\lambda = 0$
	family, corresponding to the Chern-Weil forms
    \begin{equation}
        -\frac{1}{8\pi^2}\tr(A^2) \in \Inv[2](\GL_n^+(\R))
    \end{equation}
    for $n \geq 2$ and
    \begin{equation}
         -\frac{1}{8\pi^2}A^2 \in \Inv[2](\GL_1^+(\R))
    \end{equation}
    for $n = 1$.
\end{lem}

\begin{proof}
    It is sufficient to check the universal case, where $X = \Bdot\GLp_n(\R) \times \Bdot\GLp_m(\R)$ and $E$ and
	$F$ are the tautological bundles on the first, resp.\ second factors.
    We denote the projections to the factors by $\pr_n, \pr_m\colon X \to \Bdot\GLp_{n,m}(\R)$.
    Direct sum of vector bundles induces a map
    \begin{equation}
        s\colon \Bdot\GLp_n(\R) \times \Bdot\GLp_m(\R) \to \Bdot\GLp_{n+m}(\R).
    \end{equation}
    The differential Whitney sum formula \eqref{diffWhitney} reads
    \begin{equation}
    \label{eq_to_check}
        \pr_1^*(\hat{p}_1^\lambda) + \pr_2^*( \hat{p}_1^\lambda) = s^*( \hat{p}_1^\lambda).
    \end{equation}
    We can check this under the injective map 
    \begin{equation}
	    t: H^4(\Bdot\GLp_n(\R) \times \Bdot\GLp_m(\R); \Z(2)) \to \Inv[2](\GL_m^+(\R)\times\GL_n^+(\R)).
	\end{equation}
    We interpret an element $A \in \gl_n^+(\R) \oplus \gl_m^+(\R)$ as a block diagonal matrix 
    $A_n \oplus A_m \in \gl_{n+m}^+(\R)$.
    Plugging in the polynomials corresponding to $\hat{p}_1^\lambda$ (see~\cref{GLn_off_diag}),~\eqref{eq_to_check} becomes
    \begin{equation}
        \lambda \left(\tr(A_n)^2 + \tr(A_m)^2\right) - \frac{1}{8 \pi^2} \left( \tr(A_n^2) + \tr(A_m^2) \right) =
        \lambda \tr(A_n \oplus A_m)^2 - \frac{1}{8 \pi^2}\tr((A_n \oplus A_m)^2),
    \end{equation}
    which forces $\lambda=0$.
\end{proof}
We will denote this distinguished family by $\hat{p}_1$.

\section{Cocycles for off-diagonal differential lifts}
\label{sec:cocycle}
    In this section we compute explicit cocycles for the differential lifts of the first Pontryagin class for $G =
\GL_1^+(\R) = \R^\times_+$. We will identify $\R^\times_+$ with $\R$ via the natural logarithm, inverse
to $\exp\colon\R \to \R^\times_+$.

Recall from~\cref{sec:cocycle} that that there is a one-dimensional vector space of lifts of
$p_1$ for $\GL_1^+(\R)$ (\cref{lineofliftsforGLp1}), with a distinguished element $\hat{p}_1$ 
whose image under the isomorphism
$H^4(\BdotG;\Z(2))\to\Inv[2](G)$ is $-\frac{1}{8\pi^2}A^2$ (\cref{distinguishedlifts}).
We would like to find a cocycle representing $\hat{p}_1 \in H^4(\Bdot \R; \Z(2))$; by scaling, this will give us
cocycle representatives for the entire line of differential lifts of $p_1$.

If $M$ is a paracompact manifold, possibly infinite-dimensional, then $M$ admits a partition of unity, so
$\Omega^i$ is acyclic on $M$. This implies the following well-known lemma:
\begin{lem}
    \label{OmegaCohomologyOfBdotG}
    Let $G$ be a paracompact Lie group, possibly infinite-dimensional.
    Then the cohomology groups $H^*(\BdotG; \Omega^i)$ are computed
    by the cochain complex 
    \begin{equation}
        \begin{tikzcd}
	        \Omega^i(\BdotG): \Omega^i(\point)  
	        \ar[r] & \Omega^1(G)  \ar[r]  & \Omega^i(G\times G) 
            \ar[r] &\Omega^i(G\times G\times G) \dots.
        \end{tikzcd}
    \end{equation}
\end{lem}
All finite-dimensional Lie groups are paracompact, and so is $\Gamma$~\cite[Remark 4.11]{Pal21}.

Consider the following factorization of the exterior derivative $\d: \Omega^1 \to \Omega^2$:
\begin{equation}
\label{om1_to_om2}
    \Omega^1 \overset{\iota}{\longrightarrow} \Z(2)[2] \overset{\d_1}{\longrightarrow} \Omega^2,
\end{equation}
where $\iota$ is the inclusion of $\Omega^1$ into the complex $\Z(2)[2]$, and $\d_1$ is $\d$ on $\Omega^1$ and
$0$ on the other terms in $\Z(2)$. We will also abuse notation and let $\d$, $\d_1$, and $\iota$ denote the maps on
cohomology induced by these maps between complexes of sheaves.

\begin{lem}
\label{pullback_to_diff_forms}
    The maps of sheaves in~\eqref{om1_to_om2} induce a triangle of isomorphisms 
    between three real lines
    \begin{equation}
        \begin{tikzcd}
        	& {H^4(\Bdot\R;\Z(2))} \\
        	{H^2(\Bdot\R;\Omega^1)} && {H^2(\Bdot\R;\Omega^2)}.
        	\arrow["\d", from=2-1, to=2-3]
        	\arrow["\iota", from=2-1, to=1-2]
        	\arrow["\d_1", from=1-2, to=2-3]
        \end{tikzcd}
    \end{equation}
\end{lem}
\begin{proof}
    It suffices to prove two of the above maps are isomorphisms and that
    one of the groups is abstractly isomorphic to $\R$. 
    We first show $\iota$ is an isomorphism and then show $\d$ is a nonzero
    map $\R \to \R$.

    We begin by proving $H^k(\Bdot\R;\Z(1))=0$ for $k > 2$. This follows from
    \begin{align*}
        &H^k(\Bdot\R;\Z)=H^k(B\R;\Z)=H^k(\point;\Z)\overset{k>0}{=} 0, \\
        &H^k(\Bdot\R;\Omega^0)=H^k_{\Lie}(\fR;\R) \iso H^k_{\Lie}(\R;\R)\overset{k>1}{=}0,
    \end{align*}
    and the long exact sequence
    \begin{equation}
        \cdots \to H^k(\Bdot\R; \Omega^0) \to H^{k+1}(\Bdot\R; \Z(1)) \to 
        H^{k+1}(\Bdot\R;\Z) \to H^{k+1}(\Bdot\R;\Omega^0) \to \cdots.
    \end{equation}
    Now consider the fiber sequence of coefficient sheaves
    \begin{equation}
        \begin{tikzcd}
            \Omega^1 \rar["\iota"] & \Z(2)[2] \rar & \Z(1)[2].
        \end{tikzcd}
    \end{equation}
    The associated long exact sequence in cohomology shows 
    \begin{equation}
    \label{iota_is_iso}
        \iota: H^2(\Bdot\R; \Omega^1) \isoto H^2(\Bdot\R;\Z(2)[2]) = H^4(\Bdot\R;\Z(2)) 
    \end{equation}
    is an isomorphism, because both $H^1(\Bdot\R; \Z(1)[2])=H^3(\Bdot\R; \Z(1))$ and $H^2(\Bdot\R;\Z(1)[2])=H^4(\Bdot\R;\Z(1))$ vanish.

    Using the explicit model for $H^2(\Bdot\R;\Omega^i)$ given in~\cref{OmegaCohomologyOfBdotG},
    it is straightforward to show that $x_1 \ud x_2 \in \Omega^1(\R^2)$ is a nontrivial
    $\Omega^1$-valued 2-cocycle over $\Bdot\R$
    Similarly, its differential $\ud x_1 \ud x_2 \in \Omega^2(\R^2)$ is nontrivial in $H^2(\Bdot\R;\Omega^2)$. 
    Denote the Lie algebra of $\R$ by $\fR \iso \R$. 
    Bott's theorem (\cref{botts_theorem}) implies
    \begin{align}
        H^2(\Bdot \R; \Omega^1) &\cong H^1_{\mathit{sm}}(\R; \fR^\vee) \iso \Hom(\R, \R) = \R\\
        H^2(\Bdot \R; \Omega^2) &\cong H^0_{\mathit{sm}}(\R; \Sym^2 \fR^\vee) = (\Sym^2(\fR^\vee))^\R =
    	\Sym^2(\R) \iso \R,
    \end{align}
    so the cocycles $[x_1 \ud x_2] \in H^2(\Bdot\R;\Omega^1)$ and its image under $\d$, $[\ud x_1 \ud x_2] \in H^2(\Bdot\R;\Omega^2)$,
    generate the cohomology groups they live in. This shows that indeed, $\d$ is a nonzero map
    of real lines.
\end{proof}

\begin{prop} \label{prop:cocycle}
The preimage of the distinguished class $\hat{p}_1 \in H^4(\Bdot \R; \Z(2))$ under the isomorphism
$\iota:H^2(\Bdot \R; \Omega^1) \isoto H^4(\Bdot \R; \Z(2))$ from \cref{pullback_to_diff_forms}
has a cocycle representative $\frac{1}{8\pi^2} x_1\ud x_2$ in $\Omega^1(\R^2)$.
\end{prop}
\begin{proof}
The proof boils down to computing a proportionality factor coming from the standard normalization of 
Chern-Weil theory.
Consider diagram~\eqref{omega1_to_omega2_big_diagram}: in \cref{pullback_to_diff_forms}, we saw that $\d\colon
H^2(\Bdot\R;\Omega^1)\to H^2(\Bdot\R;\Omega^2)$ factors as $\d_1\circ\iota$, and in
\cref{botts_theorem,botts_theorem2}, we saw that
$\d_1\colon H^4(\Bdot\R;\Z(2))\to H^2(\Bdot\R;\Omega^2)$ further factors through $H^4(\Bdot\R;\R(2))$ and
$\Inv[2](\R)$.

\begin{equation}
    \label{omega1_to_omega2_big_diagram}
\begin{tikzcd}[column sep=2em, row sep=1em]
	{H^2(\Bdot\R;\Omega^1)} & {H^4(\Bdot\R;\Z(2))} & {H^4(\Bdot\R;\R(2))} & {\Inv[2](\R)} & {H^2(\Bdot\R;\Omega^2)} \\
	{\frac{1}{8\pi^2} x_1\ud x_2} & {} & {} && {\frac{1}{8\pi^2}\ud x_1 \ud x_2} \\
	& {\hat{p}_1} & {} & {-\frac{1}{8\pi^2}A^2} & {\frac{1}{8\pi^2}\ud x_1\ud x_2}
	\arrow["\iota", from=1-1, to=1-2]
	\arrow[from=1-2, to=1-3]
	\arrow["{\eqref{botts_theorem2}}" swap, from=1-3, to=1-4]
	\arrow["{\eqref{2q_bott}}" swap, color={rgb,255:red,214;green,92;blue,92}, from=1-4, to=1-5]
	\arrow["\d", shift left=2, curve={height=-18pt}, from=1-1, to=1-5]
	\arrow[maps to, from=3-2, to=3-4]
	\arrow["{?}", color={rgb,255:red,214;green,92;blue,92}, maps to, from=3-4, to=3-5]
	\arrow[maps to, from=2-1, to=2-5]
\end{tikzcd}
\end{equation}

It thus suffices to show that the Bott isomorphism~\eqref{2q_bott} indeed sends $-\frac{1}{8\pi^2}A^2\mapsto \frac{1}{8\pi^2}\ud x_1\ud x_2$. 
In that case, commutativity of the maps in diagram~\eqref{omega1_to_omega2_big_diagram} implies that $\iota$ maps
$-\frac{1}{8\pi^2}A^2 \mapsto \hat{p}_1$.
The Bott isomorphism is normalized to match Chern-Weil theory (see~\cref{botts_theorem2}).
The maps fit together as indicated in diagram~\eqref{ChernWeilAndBottDiagram}, which is natural in $G$.
\begin{equation}
\label{ChernWeilAndBottDiagram}
\begin{tikzcd}
	{H^{2n}(\BdotG;\R(n))} & {H^n(\BdotG;\Omega^n)} \\
	{H^{2n}(BG;\R)} & {\Inv[n](G)}
	\arrow[from=1-1, to=2-1]
	\arrow["{\mathit{Bott}~\eqref{2q_bott}}"', from=2-2, to=1-2]
	\arrow[from=1-1, to=1-2]
	\arrow["\ChW", from=2-2, to=2-1]
	\arrow["\phi", from=2-2, to=1-1]
\end{tikzcd}
\end{equation}

It suffices to consider $G = \U_1$.
We determine the normalization on the first Chern class $c_1$ in cohomological degree 2. 
Then we use multiplicativity of the Chern-Weil and Bott maps 
to deduce the normalization in degree 4, where $\hat{p}_1$ lives.
The groups containing $c_1 \in H^2(B\U_1;\Z)$ and its unique 
lift (provided by~\cref{compact_off_diag}) $\hat{c}_1 \in H^2(\Bdot\U_1;\Z(1))$ 
fit into the relevant diagram as follows:
\begin{equation}
\label{U1ChernWeilAndBottDiagram}
\begin{tikzcd}
	{H^{2}(\Bdot\U_1;\Z(1))} & {H^{2}(\Bdot\U_1;\R(1))} & {H^1(\Bdot\U_1;\Omega^1)} \\
	{H^{2}(B\U_1;\Z)} & {H^{2}(B\U_1;\R)} & {\Inv[1](\U_1)}.
	\arrow[hook, from=1-1, to=1-2]
	\arrow[hook, from=2-1, to=2-2]
	\arrow["\cong", from=1-1, to=2-1]
	\arrow["\cong", from=1-2, to=2-2]
	\arrow["{\mathit{Bott}~\eqref{2q_bott}}"', from=2-3, to=1-3]
	\arrow[from=1-2, to=1-3]
	\arrow["\ChW", from=2-3, to=2-2]
	\arrow["\phi", from=2-3, to=1-2]
\end{tikzcd}
\end{equation}
We identify the Lie algebra of $\U_1$ as $\fu_1 = i\R$ via the exponential map
$\exp(2 \pi i\bl)\colon i\R \to \U_1$.
Chern-Weil theory associates to $c_1$ the polynomial 
$\frac{1}{2 \pi i}\mathrm{id}\colon i\R \to \R$. 
The class $c_1$ generates $H^2(B\U_1;\Z)$, and thus its lift $\hat{c}_1$ generates
$H^2(B\U_1;\Z(1)) \simeq H^1(B\U_1;\T)$. The group $H^1(B\U_1;\T)$ is the group
of homomorphisms $\U_1 \to \T$ (this is a general fact about Lie
group cohomology, see e.g.~\cite{WW15}).
Under this identification, $\hat{c}_1$ corresponds to the 
homomorphism $\chi_1\colon \U_1 \xrightarrow{\mathrm{id}} \U_1 = \T$.

The composition of maps in the top row of~\eqref{U1ChernWeilAndBottDiagram}
is induced by the map of sheaves $\d_0\colon \Z(1) \to \Omega^1[-1]$, given by $\d$ on
the $\Omega^0$-term in $\Z(1)$, and $0$ on the other terms.
This map factors through the sheaf $\T$ as
\begin{equation}
    \begin{tikzcd}[column sep=4em]
        \Z(1) \rar["\exp(2\pi i -)"] & \T[-1] \rar["\frac{1}{2\pi i} \ud \log(-)"] & \Omega^1[-1].
    \end{tikzcd}
\end{equation}
The image of $\hat{c}_1$ in
$H^1(\Bdot\U_1;\Omega^1)$ is the image of $\chi_1$ under the latter map.
We pick a representing 1-form
\begin{equation}
    \ud \theta \coloneqq \frac{1}{2\pi i} \ud \log(\chi_1) \in \Omega^1(\U_1).
\end{equation}

Now we use naturality of the Bott isomorphism to conclude the normalization for the group $\R$.
We identify the Lie algebra of $\R$ as $\fR = \R$, with exponential given by the
identity map ${\mathrm{id}}_\R:\R \to \R$.
The covering map $\exp(2\pi i-): \R \to \U_1$ induces the map $2\pi i$ on Lie algebras:
$\d_{e} \exp(2 \pi i -): \fR = \R \xrightarrow{2 \pi i} i\R = \fu_1$. 
As a result, the induced map on the degree-one polynomials on the Lie algebras is the transpose
of this linear map, $2 \pi i: \fu_1^\vee = i\R \to \R = \fR^\vee$.
This is the left-hand map in the naturality square for the Bott isomorphism:
\begin{equation}
   \begin{tikzcd}
	\Inv[1](\U_1)=i\R^\vee=i\R & H^1(\Bdot\U_1; \Omega^1) \\
	\Inv[1](\R)=\R^\vee=\R & H^1(\Bdot\R; \Omega^1).
	\arrow[from=1-2, to=2-2]
	\arrow["\mathit{Bott}", from=1-1, to=1-2]
	\arrow["\mathit{Bott}", from=2-1, to=2-2]
	\arrow[from=1-1, to=2-1]
\end{tikzcd}
\end{equation}
We compute the right-hand map $H^1(\Bdot\U_1;\Omega^1) \to H^1(\Bdot\R;\Omega^1)$ by pulling
back the representing differential forms. The form $\ud \theta \in \Omega^1(\U_1)$ pulls back to 
the differential form $\ud x \in \Omega^1(\R)$, the unique right-invariant differential form 
on $\R$ whose restriction to the Lie algebra is the identity functional 
${\ud x}_e:\fR = \R \xrightarrow{\mathrm{id}_\R} \R$.
The preimage of $\ud \theta$ under the Bott isomorphism is the linear functional
$\frac{1}{2 \pi i} \ud \theta_e: \fu_1 = i\R \xrightarrow{\frac{1}{2\pi i}} \R$.
In summary,
\begin{equation}
   \begin{tikzcd}[arrows={|->}]
	\frac{1}{2\pi i}\ud\theta_e & \d\theta \\
	\ud x_e & \ud x,
	\arrow[from=1-2, to=2-2]
	\arrow["\mathit{Bott}", from=1-1, to=1-2]
	\arrow["\mathit{Bott}", from=2-1, to=2-2]
	\arrow[from=1-1, to=2-1]
\end{tikzcd}
\end{equation}
and so the linear functional $\d x_e: \fR \to \R$ is sent to to $\d x \in
\Omega^1(\R)$. Since the Bott isomorphism is multiplicative, we can calculate the
image of $\d x^2$ using the cup product structure on the bigraded ring 
$H^{\ast}(\Bdot\R,\Omega^{\ast})$.\footnote{To clarify, the two gradings are $H^*$ and $\Omega^*$; the simplicial
direction in $\Bdot\R$ does not come into play here.}
This is the classic formula for the cup product of \v Cech cochains with values
in a complex of sheaves with ring structure. In our case, the \v Cech cover is the map 
$\point \to \point/\R$ (with \v Cech nerve the simplicial manifold $\Bdot\R$),
and the complex of sheaves is $\Omega^\bullet$ (with ring structure
given by the wedge of differential forms).
The cochain $\ud x \in \Omega^1(\R)$ lives in bidegree $(1,1)$. 
Its square $\d x^2$ has bidegree $(2,2)$ and is given
by $- p_1^\ast(\d x) p_2^\ast(\d x) = -\d x_1 \ud x_2  \in \Omega^2(\R \times \R)$ 
(see~\cite[\href{https://stacks.math.columbia.edu/tag/01FP}{Tag 01FP}]{stacks-project} 
or~\cite[Expos\'e XVII]{SGA4} for the definition of the cup product).
The distinguished Pontryagin class thus maps to
$-\frac{1}{8\pi^2}{\ud x_e}^2 \mapsto \frac{1}{8\pi^2} \ud x_1 \ud x_2 \in \Omega^2(\R^2)$.
\end{proof}

\section{Computing the transgression}
\label{sec:main_statement}
    Recall $\GL_1^+(\R)=\R_+^\times$ and $\Gamma=\Diff$. 
We identify $\GL_1^+(\R)$ with $\R$ via the isomorphism $\log: \R_+^\times \to \R$.
Let $F = \mathrm{Fr}^+(S^1) \to S^1$ be the oriented frame bundle of $S^1$. This is a trivial $\GL_1^+(\R)$-bundle:
it is canonically identified with $S^1\times\R_+^\times \to S^1$.
The action of $\Gamma$ on $S^1$ is orientation-preserving, and thus lifts
to an action on $F$.
This $\Gamma$-action commutes with the $\R$-action on $F$. 
We think of $\Gamma$ as acting on $F$ on the left and $\R$ as acting on the right. 
The double quotient stack $\GFR$ fits into the diagram
\begin{equation}
\label{eqn:transgression_space_diagram}
\begin{tikzcd}
\GFR \arrow[r, "q"] \arrow[d, "p"]
& \Bdot \R  \\
\Bdot\Gamma.
& 
\end{tikzcd}
\end{equation}
Note that $F/\R = S^1$. Under this identification, 
$\GFR = \Gamma \backslash S^1 \to \Bdot\Gamma$ is the tautological 
oriented $S^1$-bundle over $\Bdot\Gamma$.
We will compute the transgression map 
$p_\ast q^\ast: H^4(\Bdot \R; \Z(2)) \to H^3(\Bdot\Gamma; \Z(1))$ associated to~\eqref{eqn:transgression_space_diagram}. The pushforward $p_\ast$
is fiber integration in an $S^1$-bundle, so we denote it by $\int_{S^1}$.

In \S\ref{sec:off_diagonal}, we constructed differential lifts of the first Pontryagin class in
$H^4(\Bdot\R;\Z(2))$, and in \cref{lem:Tn_equiv_Zn}, we saw that $H^3(\Bdot\Gamma;\Z(1)) = H^2(\Bdot\Gamma;
\underline{\T})$ classifies Fréchet Lie group central extensions of $\Gamma$ by $\T$. Therefore applying this
transgression map turns a differential lift of $p_1$ into a central extension of $\Gamma$ by $\T$.

\begin{theorem}
\label{main_thm}
The transgression homomorphism 
\begin{equation}
    \int_{S^1} \circ \ q^\ast\colon H^4(\Bdot\R; \Z(2)) \longrightarrow H^2(\Bdot\Gamma; \Z(1))
\end{equation}
lands isomorphically in the Virasoro family $\Vir \subset \CExt_\R(\Gamma) \cong H^2(\Bdot\Gamma; \Z(1))$.
It induces an isomorphism
\begin{equation}
        \Inv[2](\R) \longrightarrow \Vir
\end{equation}
and takes the distinguished first Pontryagin class $\hat p_1$ (\cref{distinguishedlifts}) to the Virasoro central extension $\tGam_{-12}$. That is, the
Virasoro group obtained by transgressing $\hat p_1$ has central charge $-12$.
\end{theorem}
\begin{rem}
Recall from \cref{bosonic} that we choose our normalization so that $\GpVir_1$ has central charge 1, that is, it
is the central extension that acts on the bosonic string CFT.
\end{rem}

We begin by reducing to a computation on differential forms. By \cref{prop:cocycle}, the map $H^2(\Bdot \R,
\Omega^1) \to H^4(\Bdot \R; \Z(2))$ is an isomorphism, and the preimage of $\hat{p}_1$ is $\frac{1}{8\pi^2} x_1 \ud x_2 \in \Omega^1(\R^2)$. Functoriality of pullback and integration imply the following commutative diagram (\cref{fibcomm}): 
\begin{equation}
\begin{tikzcd} \label{reduction-to-diff-form}
    H^2(\Bdot\R;\Omega^1) \ar[r,"q^\ast"] \ar[d, "\cong"] &
    H^2(\GFG; \Omega^1) \ar[r,"{\int_{S^1}}"] \ar[d] &
    H^2(\Bdot \Gamma;\Omega^0) \ar[d, "\phi"] \\
    H^4(\Bdot\R;\Z(2)) \ar[r,"q^\ast"] &
    H^4(\GFG; \Z(2)) \ar[r,"{\int_{S^1}}"] &
    H^3(\Bdot \Gamma;\Z(1)).
\end{tikzcd}
\end{equation}
The composition
\begin{equation}
	H^2(\Bdot\Gamma;\Omega^0)\overset\phi\longrightarrow H^3(\Bdot\Gamma;\Z(1))\overset\cong\longrightarrow
	H^2(\Bdot\Gamma; \underline\T)
\end{equation}
is the map induced by 
\begin{equation}
    \exp(2 \pi i -): \Omega^0 = \underline\R \to \underline \T.
\end{equation}
Therefore we can compute the transgression using the top line of the diagram~\eqref{reduction-to-diff-form}, where
we have the advantage of working with differential forms, and then exponentiate to get back to
$H^3(\Bdot\Gamma;\Z(1))$.\footnote{This last step is a map $H^2(\Bdot\Gamma;\underline\R)\to
H^2(\Bdot\Gamma;\underline\T)$ and thus can be interpreted as taking a central extension of $\Gamma$ by $\R$ and
building a central extension by $\T$.} In particular, to prove \cref{main_thm}, it suffices to show the following:
\begin{prop}
\label{main_prop}
The transgression map 
$\int_{S^1} \circ \ q^\ast: H^2(\Bdot\R;\Omega^1) \to H^2(\Bdot\Gamma;\Omega^0)$ maps the class $[x_1\ud x_2]$
to the class of the central extension $\R \to \GpVir_\R \to \Gamma$ corresponding to the unnormalized $\R$-valued Bott-Thurston 
cocycle\footnote{This ``unnormalized'' cocycle corresponds to the case $\lambda=-96 \pi^2$ in~\eqref{real-cocycle}.} 
\begin{equation}
\label{bott_R}
   B_\R(\gamma_1, \gamma_2)\coloneqq \int_{S^1} \log(\gamma_1'\circ\gamma_2) \ud(\log(\gamma_2')).
\end{equation}
\end{prop}
We compute the map explicitly on cocycles, using simplicial presentations of the spaces and stacks involved. Recall that elements of $\Omega^1(\R^2)$ are cochains for $H^2(\Bdot\R;\Omega^1)$ with respect to the following simplicial presentation of $\Bdot \R$:
\begin{equation}
\begin{tikzcd}[arrows={-Stealth}]
	\point & 
	\ar[l,shift left=0.15em] \ar[l,shift right=0.15em]  \R &
	\ar[l,shift left=0.3em] \ar[l] \ar[l,shift right=0.3em] \R\times \R &
	\ar[l,shift left=0.45em] \ar[l,shift left=0.15em]
	\ar[l,shift right=0.15em] \ar[l,shift right=0.45em] \R\times \R\times \R \cdots.
\end{tikzcd}
\end{equation}

We compute the pullback of $x_1\ud x_2 \in \Omega^1(\R^2)$ to $\GFR$ using the following simplicial 
presentation of the map $\GFR \to \Bdot\R$:
\begin{equation}
\label{pullback_big_diagram}
\begin{tikzcd}[arrows={-Stealth}]
	\ar[d] \GF & 
	\ar[d] \ar[l,shift left=0.15em] \ar[l,shift right=0.15em]  \GF \times \R&
	\ar[d] \ar[l,shift left=0.3em] \ar[l] \ar[l,shift right=0.3em] \GF\times \R \times \R &
	\ar[d] \ar[l,shift left=0.45em] \ar[l,shift left=0.15em]
	\ar[l,shift right=0.15em] \ar[l,shift right=0.45em] \GF\times \R\times \R \times \R \dots \\
	\point & 
	\ar[l,shift left=0.15em] \ar[l,shift right=0.15em]  \R &
	\ar[l,shift left=0.3em] \ar[l] \ar[l,shift right=0.3em] \R\times \R &
	\ar[l,shift left=0.45em] \ar[l,shift left=0.15em]
	\ar[l,shift right=0.15em] \ar[l,shift right=0.45em] \R\times \R\times \R \dots
\end{tikzcd}
\end{equation}
where the vertical maps simply forget the first factor.
As a result, the image of $x_1 \ud x_2$ under the pullback map 
$p^*: \Omega^1(\Bdot\R) \rightarrow \Omega^1((\Gamma \backslash \F)/\R_\bullet)$ 
is constant along the factor $\Gamma \backslash F$. We denote the 
cocycle $p^*(x_1 \ud x_2)$ by the same symbol, 
$x_1 \ud x_2 \in \Omega^1(\Gamma \backslash \F \times \R^2)$. 

To compute the pushforward of this class along the projection
$\GFR \to \Bdot\Gamma$, we pick a different presentation
of $\GFR$. Instead of resolving the $\R$-action, we resolve the $\Gamma$-action:
\begin{equation}
\begin{tikzcd}[arrows={-Stealth}]
	F/\R & 
	\ar[l,shift left=0.15em] \ar[l,shift right=0.15em]  \Gamma \times F/\R &
	\ar[l,shift left=0.3em] \ar[l] \ar[l,shift right=0.3em] \Gamma \times \Gamma \times F/\R &
	\ar[l,shift left=0.45em] \ar[l,shift left=0.15em]
	\ar[l,shift right=0.15em] \ar[l,shift right=0.45em] \Gamma \times \Gamma \times \Gamma \times F/\R \dots.
\end{tikzcd}
\end{equation}

The map $q\colon \GFR \rightarrow \Bdot \Gamma$ admits a presentation by the simplicial map
\begin{equation}
\label{pushforward_big_diagram}
\begin{tikzcd}[arrows={-Stealth}]
	\dar F/\R & 
	\dar \ar[l,shift left=0.15em] \ar[l,shift right=0.15em]  \Gamma \times F/\R &
	\dar \ar[l,shift left=0.3em] \ar[l] \ar[l,shift right=0.3em] \Gamma \times \Gamma \times F/\R &
	\dar \ar[l,shift left=0.45em] \ar[l,shift left=0.15em]
	\ar[l,shift right=0.15em] \ar[l,shift right=0.45em] \Gamma \times \Gamma \times \Gamma  \times F/\R \dots. \\
	\point & 
	\ar[l,shift left=0.15em] \ar[l,shift right=0.15em]  \Gamma &
	\ar[l,shift left=0.3em] \ar[l] \ar[l,shift right=0.3em] \Gamma\times \Gamma &
	\ar[l,shift left=0.45em] \ar[l,shift left=0.15em]
	\ar[l,shift right=0.15em] \ar[l,shift right=0.45em] \Gamma\times \Gamma\times \Gamma \dots
\end{tikzcd}
\end{equation}
which is a level-wise $S^1$-fibration.
As a result, $p_*\colon \Omega^1(\Gamma \backslash (F/\R)_\bullet) \rightarrow
\Omega^1(\Bdot\Gamma)$ may be computed by integrating over $S^1$ level-by-level.

In~\cref{main_lemma}, we will prove that $x_1 \ud x_2 \in \Omega^1(\Gamma \backslash \F \times \R^2)$ 
is cohomologous to the cocycle given by the integrand of the $\R$-valued Bott-Thurston 
cocycle~\eqref{bott_R}, $\log(\gamma_1'\circ\gamma_2) \ud(\log(\gamma_2'))$. This is enough to imply
\cref{main_prop}: we saw above using~\eqref{pullback_big_diagram} that $x_1\ud x_2\in\Omega^1(\Gamma\backslash
(F/\R)_\bullet)$ represents the cohomology class $q^*([x_1\ud x_2])$. \Cref{main_lemma} will show that
$\log(\gamma_1'\circ\gamma_2)\ud(\log(\gamma_2'))$ represents the same cohomology class; thus we can use the
latter cocycle to compute the pushforward, and using~\eqref{pushforward_big_diagram}, the pushforward will be
$\int\log(\gamma_1'\circ\gamma_2)\ud(\log(\gamma_2'))$, proving \cref{main_prop}. So all we have left to do is
prove \cref{main_lemma}.

The challenge is transporting our cocycle from the first simplicial presentation of $\GFR$ (where we resolve in the
$\R$-direction) to the second (where we resolve in the $\Gamma$-direction). To do so, we will chase it across the
double complex associated to the bisimplicial manifold $\GFR_{\bullet, \bullet}$ obtained by resolving both of
these objects. Specifically, $\GFR_{p,q} = \Gamma^p \times F \times \R^q$:
\begin{equation}
\begin{gathered}
\begin{tikzcd}[column sep=3em,row sep=3em, arrows={-Stealth}]
    &
	\vdots
	\ar[d,shift left=0.45em]  \ar[d,shift left=0.15em]
	\ar[d,shift right=0.15em] \ar[d,shift right=0.45em] &
	\vdots
	\ar[d,shift left=0.45em]  \ar[d,shift left=0.15em]
	\ar[d,shift right=0.15em] \ar[d,shift right=0.45em] &
	\vdots
	\ar[d,shift left=0.45em]  \ar[d,shift left=0.15em]
	\ar[d,shift right=0.15em] \ar[d,shift right=0.45em] 
	\\
    \cdots
	\ar[r,shift left=0.45em]  \ar[r,shift left=0.15em]
	\ar[r,shift right=0.15em] \ar[r,shift right=0.45em] &
	\Gamma^{\times 2} \times F \times \R^{\times 2}
	\ar[d,shift left=0.3em]  \ar[d] \ar[d,shift right=0.3em]
	\ar[r,shift left=0.3em]  \ar[r] \ar[r,shift right=0.3em] &
    \Gamma \times F \times \R^{\times 2}
	\ar[d,shift left=0.3em]  \ar[d] \ar[d,shift right=0.3em]
	\ar[r,shift left=0.15em]  \ar[r,shift right=0.15em] &
	F \times \R^{\times 2}
	\ar[d,shift left=0.3em]  \ar[d] \ar[d,shift right=0.3em]
	\\
    \cdots
	\ar[r,shift left=0.45em]  \ar[r,shift left=0.15em]
	\ar[r,shift right=0.15em] \ar[r,shift right=0.45em] &
	\Gamma^{\times 2} \times F \times \R
	\ar[d,shift left=0.15em]  \ar[d,shift right=0.15em]
	\ar[r,shift left=0.3em]  \ar[r] \ar[r,shift right=0.3em] &
    \Gamma \times F \times \R
	\ar[d,shift left=0.15em]  \ar[d,shift right=0.15em]
	\ar[r,shift left=0.15em]  \ar[r,shift right=0.15em] &
	F \times \R
	\ar[d,shift left=0.15em]  \ar[d,shift right=0.15em]
	\\
    \cdots
	\ar[r,shift left=0.45em]  \ar[r,shift left=0.15em]
	\ar[r,shift right=0.15em] \ar[r,shift right=0.45em] &
	\Gamma^{\times 2} \times F
	\ar[r,shift left=0.3em]  \ar[r] \ar[r,shift right=0.3em] &
    \Gamma \times F 
	\ar[r,shift left=0.15em]  \ar[r,shift right=0.15em] &
	F. 
\end{tikzcd}
\end{gathered}
\end{equation}

(The bisimplicial set is oriented this way to suggest 
the simplicial version of the transgression diagram~\eqref{eqn:transgression_space_diagram}.) 
We view $\GFR_{\bullet, \bullet}$ as a simplicial resolution of $((\Gamma \backslash F)/\R)_\bullet$, 
and separately of ${(\Gamma \backslash (F/\R))}_\bullet$, by projecting to 
the simplicial sets along $p=0$ and $q=0$, respectively.
These projections induce pullback maps 
\begin{equation}
\begin{tikzcd}
\Omega^1(\GFR_{\bullet, \bullet})
& \arrow[l, "f^*"] \Omega^1(((\Gamma \backslash F)/\R)_\bullet)  \\
\Omega^1({\Gamma \backslash (F/\R)}_\bullet) \arrow[u, "g^*"].
& 
\end{tikzcd}
\end{equation}
The degree-$3$ piece of the total complex is 
\begin{equation}
\label{deg3_decomp}
    \Omega^1(F \times \R^{\times 2}) \oplus \Omega^1(\Gamma \times F \times \R) \oplus 
    \Omega^1(\Gamma^{\times 2} \times F).
\end{equation}
We pick an identification $F \simeq S^1 \times \R_+^\times$,
such that the right action by $\R$ is given by the exponential
map $\R \to \R_+^\times$, followed by multiplication.
To describe the differential forms, it is helpful to fix some further notation. 
We use $\gamma$ to denote elements of $\Gamma$,
$(\theta, v)$ for elements of $F$, and
$x$ for elements of $\R$. The action maps are given by
\begin{align}
\begin{split}
    \Gamma \times F &\to F \\
    (\gamma, \theta, v) &\mapsto (\gamma(\theta),\gamma'(\theta) \cdot v)
\end{split}
\end{align}
and
\begin{align}
\begin{split}
    F \times \R &\to F \\
    (\theta, v, x) &\mapsto (\theta, e^x \cdot v).
\end{split}
\end{align}

We are interested in two cocycles.
\begin{enumerate}
	\item The starting point is $f^*(x_1 \ud x_2)$, which in the decomposition~\eqref{deg3_decomp} is
	\begin{equation}
		z_1\coloneqq (x_1 \ud x_2, 0, 0).
	\end{equation}
	\item Our goal is to obtain the pullback of the integrand of the Bott-Thurston cocycle~\eqref{bott_R} under $g^*$.
	In the decomposition~\eqref{deg3_decomp} this is
	\begin{equation}
		z_2\coloneqq (0,0, \log(\gamma_1'\circ\gamma_2) \ud(\log(\gamma_2'))).
	\end{equation}
\end{enumerate}
We will show these two cocycles are cohomologous in the total complex.
\begin{lem}\label{main_lemma} 
The cocycles $z_1$ and $z_2$ are cohomologous, i.e.\ their difference is a coboundary:
\begin{equation}
\label{dbeta}
	z_2 - z_1 = (-x_1 \ud x_2, 0, \log(\gamma_1'\circ\gamma_2) \ud(\log(\gamma_2'))) = d \beta,
\end{equation}
where $\beta = (-\log(v) \ud x,\log(\gamma')\ud  \log(v))$ is a  degree $2$ cocycle in the double complex, with
$-\log(v) \ud x \in \Omega^1(F \times \R)$ and $\log(\gamma')\ud  \log(v) \in \Omega^1(\Gamma \times F)$.
\end{lem}
See \cref{double_cpx_fig} for a visualization of \cref{main_lemma} and of~\eqref{dbeta} more specifically.
\begin{figure}[h!]
\begin{tikzpicture}
    \draw  node (N20) at (-6,0) {$\Omega^1(\Gamma^{\times 2} \times F)$}
           node (N10) at (-3,0) {$\Omega^1(\Gamma \times F)$}
           node (N11) at (-3,3) {$\Omega^1(\Gamma \times F \times \R)$}
           node (N01) at (0,3)  {$\Omega^1(F \times \R)$}
           node (N02) at (0,6)  {$\Omega^1(F \times \R^{\times 2})$};
    \draw [->] (N10) -- (N20);
    \draw [->] (N10) -- (N11);
    \draw [->] (N01) -- (N11);
    \draw [->] (N01) -- (N02);
    \begin{scope}[shift={(1,-0.5)}]
        \draw  node (C20) at (-6,0) {\hspace{-5em}
               $z_2 = \log(\gamma_1'\circ\gamma_2) \ud\log(\gamma_2')$}
               node (C10) at (-3,0) {$\log(\gamma')\ud  \log(v)$} 
               node (C11) at (-3,3) 
               [align=center]{\\
               $\textcolor{blue}{\log \gamma^\prime \ud x } =$ \\
               $\textcolor{red}{\log \gamma^\prime \ud x }$}
               node (C01) at (0,3)  {$-\log(v) \ud x$}
               node (C02) at (0,6)  {\quad$-z_1 = -x_1 \ud x_2$};
        \draw [|->] (C10) -- (C20);
        \draw [|->,draw=red] (C10) -- (C11);
        \draw [|->,draw=blue] (C01) -- (C11);
        \draw [|->] (C01) -- (C02);
    \end{scope}
\end{tikzpicture}
\caption{A schematic of the proof of \cref{main_lemma}: that if $\beta = (\log(v)\ud x, \log(\gamma')\ud\log(v))$ in
total degree $2$ in the double complex $\Omega^1(\GFR_{\bullet, \bullet})$, then $\d\beta = z_1 - z_2$.}
\label{double_cpx_fig}
\end{figure}

    \begin{proof}
Let $\dv$ and $\dh$ be the vertical, resp.\ horizontal, differentials in the total complex. We need to show the
following:
\begin{enumerate}
    \item $\dv (-\log(v) \ud x) = -x_1 \ud x_2$.
    \item $\dh (-\log(v) \ud x) = \log \gamma' \ud x = \dv (\log(\gamma') \ud \log(v))$.
    \item $\dh (\log(\gamma') \ud \log(v)) = \log(\gamma_1'\circ\gamma_2) \ud\log(\gamma_2')$
\end{enumerate}

First we show that 
\begin{equation}
    \dv (-\log(v) \ud x) = -x_1\ud x_2.
\end{equation}
Recall that we have three maps $\dv_0, \dv_1, \dv_2: \F \times \R^2 \rightarrow \F \times \R$, and $\dv = \sum (-1)^i
(\dv_i)^*$:
\begin{equation}
\begin{aligned}
		&\overset{\dv_0}{\rotatebox[origin=B]{20}{$\longmapsto$}}\, (\theta, e^{x_1} v, x_2)\\
	(\theta, v, x_1, x_2) &\overset{\dv_1}{\longmapsto} (\theta, v, x_1 + x_2)\\
		&\overset{\dv_2}{\rotatebox[origin=B]{-20}{$\longmapsto$}}\, (\theta, v, x_1).
\end{aligned}
\end{equation}
Thus 
\begin{equation}
\begin{aligned}
    \dv (-\log(v) \ud x) 
    &= -\log(e^{x_1} v) \ud x_2 + \log(v) \ud (x_1 + x_2) - \log(v) \ud x_1 \\
    &= -x_1\ud x_2.
\end{aligned}
\end{equation}

Similarly, $\dh = \sum_i (-1)^i (\dh_i)^*$, and we want to show that
$\dh(-\log(v)\ud x) = \log\gamma'\ud x$.
In this case there are two maps $\dh_0, \dh_1\colon\Gamma\times\F\times\R\to\F\times\R$: 
\begin{equation}
	(\gamma, \theta, v, x)
	\begin{aligned}
		&\overset{\dh_0}{\rotatebox[origin=c]{10}{$\longmapsto$}} 
		(\theta, v, x)\\
		&\underset{\dh_1}{\rotatebox[origin=c]{-10}{$\longmapsto$}} 
		(\gamma(\theta), \gamma'(\theta)v, x).
	\end{aligned}
\end{equation}
Thus
\begin{equation}
\begin{aligned}
    \dh (-\log(v) \ud x) 
    &= - \log(v) \ud x + \log(\gamma' \cdot v) \ud x  \\
    &= \log(\gamma') \ud x.
\end{aligned}
\end{equation}

Next we show that $\dv(\log(\gamma')\ud\log(v)) = \log(\gamma')\ud x$.
Recall $\dv_0, \dv_1\colon \Gamma \times \F \times \R \rightarrow \Gamma \times \F$ are given by
\begin{equation}
	(\gamma, \theta, v, x)
	\begin{aligned}
		&\overset{\dv_0}{\rotatebox[origin=c]{10}{$\longmapsto$}} (\gamma, \theta, e^x v)\\
		&\underset{\dv_1}{\rotatebox[origin=c]{-10}{$\longmapsto$}} (\gamma, \theta, v),
	\end{aligned}
\end{equation}
and $\dv = (\dv_0)^* - (\dv_1)^*$, so

\begin{equation}
\begin{aligned}
    \dv (\log(\gamma') \ud\log(v)) 
    &= \log(\gamma') \ud \log(e^x \cdot v) - \log(\gamma') \ud\log(v) \\
    &= \log(\gamma') \ud x.
\end{aligned}
\end{equation}

Lastly we need to show that 
\begin{equation}
    \dh (\log(\gamma') \ud \log(v)) =  \log(\gamma_1'\circ\gamma_2) \ud\log(\gamma_2').
\end{equation}
We have three maps $\dh_0, \dh_1, \dh_2: \Gamma^2 \times \F \rightarrow \Gamma \times \F$, given by
\begin{equation}
\begin{aligned}
		&\overset{\dh_0}{\rotatebox[origin=B]{20}{$\longmapsto$}}\, 
		(\gamma_2, \theta, v)\\
	(\gamma_1, \gamma_2, \theta, v) &\overset{\dh_1}{\longmapsto} (\gamma_1 \circ \gamma_2, \theta, v)\\
		&\underset{\dh_2}{\rotatebox[origin=B]{-20}{$\longmapsto$}}\, 
		(\gamma_1, \gamma_2(\theta), \gamma_2'(\theta) v),
\end{aligned}
\end{equation}
and $\dh = (\dh_0)^* - (\dh_1)^* + (\dh_2)^*$:
\begin{subequations}
\begin{equation}
    (\dh_0) ^*(\log(\gamma') \ud\log(v)) = \log \gamma_2'(\theta) \ud \log v.
\end{equation}

\begin{equation}
\begin{aligned}
    (\dh_1) ^*(\log(\gamma') \ud\log(v)) &=  \log(\gamma_1 \circ \gamma_2)' (\theta) \ud \log v \\ 
    &= \log(\gamma_1' (\gamma_2(\theta))) \ud \log v+ \log  \gamma_2'(\theta) \ud \log v.
\end{aligned}
\end{equation}

\begin{equation}
\begin{aligned}
    (\dh_2) ^*(\log(\gamma') \ud\log(v)) &= \log  \gamma_1'(\gamma_2(\theta)) \ud (\log (\gamma_2'(\theta) v))   \\ 
    &= \log  \gamma_1'(\gamma_2(\theta)) \ud \log\ (\gamma_2'(\theta))+ \log  \gamma_1'(\gamma_2(\theta)) \ud 
    \log v.
\end{aligned}
\end{equation}

\end{subequations}
Thus
\begin{equation}
\begin{aligned}
    \dh (\log(\gamma') \ud\log(v)) &= ((\dh_0)^* - (\dh_1)^* + (\dh_2)^*) (\log(\gamma') \ud\log(v)) \\
    &= \log(\gamma_1'\circ\gamma_2) \ud\log(\gamma_2').
\end{aligned}
\end{equation}
This completes the proof of \cref{main_lemma}.
\end{proof}



\end{document}